\documentclass[12pt]{amsart}
\usepackage{amsmath,amsthm,amsfonts,amssymb,amscd,euscript}
\usepackage{enumerate,enumitem}
\usepackage{color}
\usepackage{graphics}
\usepackage{mathtools}
\usepackage{tikz}
\usepackage{tkz-graph}
\usetikzlibrary{automata,arrows}
\usetikzlibrary{shapes.geometric}
\usetikzlibrary{patterns}
\usepackage{tkz-euclide}
\usetkzobj{polygons}

\usepackage{graphicx}
\usepackage{caption}
\usepackage{subcaption}
\DeclareCaptionSubType[Alph]{figure}
\usepackage{todonotes}
\usepackage{hyperref}
\usepackage{marginnote}

\makeatletter
\newcommand{\addresshere}{%
  \enddoc@text\let\enddoc@text\relax
}
\makeatother
\input{xy}
\xyoption{all}
\newlength{\defbaselineskip}
\theoremstyle{plain}
\newtheorem{theorem}{Theorem}[section]

\newtheorem{corollary}[theorem]{Corollary}
\newtheorem{lemma}[theorem]{Lemma}

\theoremstyle{definition}

\newtheorem{definition}[theorem]{Definition}

\newtheorem{example}[theorem]{Example}
\newtheorem{question}[theorem]{Question}

\newtheorem{conjecture}[theorem]{Conjecture}

\DeclareMathOperator{\mut}{Mut}
\renewcommand{\emptyset}{\varnothing}
\newcommand{\mgsi}{\ensuremath{\fontseries\bfdefault\textit{i}}}
\newcommand{\hdef}[1]{\textbf{#1}}

\begin{document}
\author{John W. Lawson}
\address{Department of Mathematical Sciences, Durham University, Lower Mountjoy,
Stockton Road, DH1 3LE, UK}
\email{j.w.lawson@durham.ac.uk}
\thanks{JL is supported by an EPSRC PhD scholarship.}

\author{Matthew R. Mills}
\email{matthew.mills@huskers.unl.edu}
\thanks{MM is supported by University of Nebraska - Lincoln, and by NSA grant H98230-14-1-0323.}
\address{Department of Mathematics, University of Nebraska - Lincoln, Lincoln, NE, USA}

\title{Properties of minimal mutation-infinite quivers}

\begin{abstract} 
We study properties of minimal mutation-infinite quivers. In particular we show that every minimal-mutation infinite quiver of at least rank 4 is Louise and has a maximal green sequence. It then follows that the cluster algebras generated by these quivers are locally acyclic and hence equal to their upper cluster algebra. We also study which quivers in a  mutation-class have a maximal green sequence. For any rank 3 quiver there are at most 6 quivers in its mutation class that admit a maximal green sequence. We also show that for every rank 4 minimal mutation-infinite quiver there is a finite connected subgraph of the unlabelled exchange graph consisting of quivers that admit a maximal green sequence.
\end{abstract}

\maketitle

\section{Introduction}
\label{sec:in}

Cluster algebras were introduced by Fomin and Zelevinsky in~\cite{cluster1} and
have since found applications in various areas of mathematics, from
representation theory to quiver gauge theories. One of the central ideas of these
algebras is mutation of quivers and seeds.

Felikson, Shapiro and Tumarkin classified all mutation-finite quivers, those
which only admit a finite number of other quivers under repeated mutations,
in~\cite{felikson}. Such quivers either arise from triangulations of surfaces or
belong to a small number of exceptional classes. A natural extension of this
idea is the study of minimal mutation-infinite quivers, a quiver which admits an infinite number of quivers under repeated mutation, but any induced proper subquiver is mutation-finite. Felikson et al.\ showed that outside of the rank 3 case there are a finite number of minimal mutation-infinite
quivers in~\cite{felikson} and Lawson classified these quivers in~\cite{lawson}. All mutation-infinite rank 3 quivers are minimal mutation-infinite since every rank 2 quiver is mutation-finite. However, rank 3 quivers have been studied extensively~\cite{assem,bbh,llm,muller,S1} so this paper primarily focuses on minimal mutation-infinite quivers of rank at least 4. 

Maximal green sequences are another aspect of cluster algebras which have been
widely studied, with links to scattering diagrams~\cite{brustle2} as well as to
BPS states of some quantum field theories~\cite{ACCERV}. Mills determined which
mutation-finite quivers admit a maximal green sequence in~\cite{mills}. We show here that every mutation-finite quiver that appears as a subquiver of a minimal mutation-infinite quiver has a maximal green sequence and prove the same result for minimal mutation-infinite quivers.

{%
\renewcommand{\thetheorem}{\ref{thm:mmi_mgs}}%
\begin{theorem}
	Suppose $Q$ is a minimal mutation-infinite quiver of rank at least 4. Then $Q$
	has a maximal green sequence.
\end{theorem}
\addtocounter{theorem}{-1}%
}

Each cluster algebra has an associated upper cluster algebra, introduced by
Berenstein, Fomin and Zelevinsky~\cite{BFZ}, which naturally contains its
cluster algebra. In general it is not known when the upper cluster algebra is equal to
the cluster algebra. One result of~\cite{BFZ} was to show that if a quiver was acyclic and coprime then these two algebras are equal. Muller later showed that the coprime assumption was not necessary and that the result held more generally if a quiver is locally acyclic~\cite{M}. Muller and Speyer went on to identify a stronger property of quivers, called the Louise property, that implies local acyclicity as well as many other very nice properties about the cluster algebra~\cite{MS}. We show that minimal mutation-infinite quivers have this property. 
{%
\renewcommand{\thetheorem}{\ref{thm:mmi_Louise}}%
\begin{theorem}
	If $Q$ is a minimal mutation-infinite quiver of rank at least 4, then the
	quiver is Louise and its cluster algebra is locally-acyclic.
\end{theorem}
\addtocounter{theorem}{-1}%
}
By showing that all minimal mutation-infinite quivers are Louise
it follows that the cluster algebra generated by any such quiver is equal to its
upper cluster algebra.
{%
\renewcommand{\thetheorem}{\ref{cor:mmi_upper}}%
\begin{corollary}
	If $Q$ is a minimal mutation-infinite quiver of rank at least 4, then the
	cluster algebra $\mathcal{A}(Q)$ is equal to its upper cluster algebra. 
\end{corollary}
\addtocounter{theorem}{-1}%
}
Such a result would not be especially interesting if the minimal
mutation-infinite quivers of the same rank generate the same cluster algebra, so we determine
which mutation classes these quivers belong to. The different move-classes given by
the classification of minimal mutation-infinite quivers in~\cite{lawson} groups
together certain mutation-equivalent quivers and we show that most of these groups belong to
different mutation classes.

{%
\renewcommand{\thetheorem}{\ref{thm:hcs_diff_class} and~\ref{thm:diff_class}}%
\begin{theorem}
	With one exception, the move-classes of minimal mutation-infinite
	quivers with hyperbolic Coxeter simplex representatives or with double arrow
	representatives all belong to distinct mutation classes.
\end{theorem}
\addtocounter{theorem}{-1}%
}

As each move-class belongs to a unique mutation class and each quiver in the
move-class has a maximal green sequence, we explored which other quivers in the
mutation class admit a maximal green sequence. In the rank 3 case we show that the number of such quivers is bounded.

{%
\renewcommand{\thetheorem}{\ref{thm:finite_rank3_mgs}}%
\begin{theorem}
 For any rank 3 quiver there are at most 6 quivers (up to relabelling of the vertices) in its mutation class that admit a maximal green sequence. 
\end{theorem}
\addtocounter{theorem}{-1}%
}

One way to visualise the mutation class is by considering the quiver exchange graph, and we use this construction to prove that in the rank 4 case those quivers which have maximal green sequences, and are obtained from mutating a minimal mutation-infinite quiver a small number of times, form a connected
subgraph in the exchange graph.

{%
\renewcommand{\thetheorem}{\ref{thm:rank4_eg}}%
\begin{theorem}
	Let $Q$ be a minimal mutation-infinite quiver of rank 4. Then the subgraph
	$\Psi$ of the quiver exchange graph $\Gamma$ containing all quivers with
	maximal green sequences is a proper subgraph of $\Gamma$ and the connected
	component $\widehat{\Psi}$ of $\Psi$ that contains $Q$ is finite and contains
	the entire move-class of $Q$.
\end{theorem}
\addtocounter{theorem}{-1}%
}

This raises a number of questions related to the arrangement of quivers with
maximal green sequences in infinite quiver exchange graphs. In general it is not currently
known whether the subgraph of quivers with maximal green sequences is connected,
nor whether there are a finite number of quivers with maximal green sequences in
a given mutation class.

In this paper Sections~\ref{sec:mgs} and~\ref{sec:upper} remind the reader of
the basic definitions of quiver mutation, maximal green sequences, cluster
algebras and upper cluster algebras, along with a number of known results on
these topics which are used throughout this paper. Section~\ref{sec:mmi-mut} contains some results on the mutation classes of these
quivers. In Section~\ref{sec:mmi-mgs}
we show the existence of a maximal green sequence for all minimal
mutation-infinite quivers and in Section~\ref{sec:mmi-upper} show that their
cluster algebra is equal to its upper cluster algebra.
 In Section~\ref{sec:mut_class_mgs} we discuss which quivers in the mutation class of rank 3 and rank 4 minimal mutation-infinite quivers have maximal green sequences, and finally in Section~\ref{sec:conj} we present a number of conjectures
and questions which build on this work.

The authors would like to thank Kyungyong Lee and Pavel Tumarkin for helpful comments that improved the presentation of this work. 

\section{Quivers and Maximal green sequences}\label{sec:mgs}

Maximal green sequences were first studied by Keller in~\cite{keller2}. In the
following we use the conventions established by Br\"ustle, Dupont and P\'erotin in \cite{brustle}.

\begin{definition}
A \hdef{(cluster) quiver} is a directed graph with no loops or 2-cycles. An \hdef{ice quiver} is a pair $(Q, F)$ where $Q$ is a quiver and $F$ is a subset of the vertices of $Q$ called \hdef{frozen vertices;} such that there are no edges between frozen vertices. If a vertex of $Q$ is not frozen it is called \hdef{mutable}. For convenience, we assume that the mutable vertices are labelled $\{1,\ldots, n\}$, and frozen vertices are labelled by $\{n+1, \ldots, n+m\}$. 
\end{definition}

For a quiver $Q$ we denote the set of vertices as $Q_0$ and the set of edges
$Q_1$.

\begin{definition}\label{def:mutation}
Let $(Q,F)$ be an ice quiver, and $k$ a mutable vertex of $Q$. The \hdef{mutation} of $(Q,F)$ at vertex $k$ is denoted by $\mu_k$, and is a transformation $(Q,F)$ to a new ice quiver $(\mu_k(Q),F)$ that has the same vertices, but making the following adjustment to the edges:  \begin{enumerate}
\item For every 2-path $i \rightarrow k \rightarrow j$, add a new arrow $i \rightarrow j$. 
\item Reverse the direction of all arrows incident to $k$. 
\item Delete a maximal collection of 2-cycles created during the first two steps, and any arrows between frozen vertices. 
\end{enumerate} \end{definition} 
Mutation at a vertex is an involution, and gives an equivalence relation. We
define $\mut(Q)$ to be the equivalence class of all quivers under this relation that can be obtained from $Q$ by a sequence of mutations. The mutation class of minimal mutation-infinite quivers are the object of study in Section~\ref{sec:mut_class_mgs} so we introduce a way to visualize this set.

\begin{definition}
The \hdef{unlabelled exchange graph} of a quiver $Q$, denoted $\Gamma(Q),$ is the graph with vertex set $\mut(Q)$ and an edge between two vertices if and only if the corresponding quivers are related by a single mutation.
\end{definition}

Given a proper subset $V$ of vertices of a quiver $Q$, the \hdef{induced subquiver}
$Q[V]$, is the quiver constructed from $Q$ by removing all vertices not in $V$
and all edges incident to these removed vertices. As such, $Q[V]$ has vertex set
$V$ and edges in $Q[V]$ are those edges in $Q$ between pairs of vertices in $V$.
If $V$ is a proper subset of vertices then $Q[V]$ is said to be a \hdef{induced proper subquiver}. 

If $\mut(Q)$ is finite (resp., infinite), then $Q$ is \hdef{mutation-finite} (resp., \hdef{mutation-infinite}). If $Q$ is mutation-infinite and every induced proper subquiver is mutation-finite then $Q$ is \hdef{minimal mutation-infinite}.

\begin{definition}{\cite[Definition 2.4]{brustle}}
The \hdef{framed quiver} associated with a quiver $Q$ is the ice quiver $(\hat{Q},Q_0')$ such that:

$$Q_0'=\{i'\text{ }|\text{ }i\in Q_0\}, \hspace{.6cm} \hat{Q}_0 = Q_0 \sqcup Q_0'$$
$$\hat{Q}_1 = Q_1 \sqcup \{i \to i'\text{ }|\text{ }i \in Q_0\}$$
\end{definition}

Since the frozen vertices of the framed quiver are so natural we will simplify the notation and just write $\hat{Q}$. Now we must discuss what is meant by red and green vertices.

\begin{definition}{\cite[Definition 2.5]{brustle}}\label{def:green1}
Let $R \in \mut(\hat{Q})$. \\A mutable vertex $i \in R_0$ is called \hdef{green} if $$\{j'\in Q_0'\text{ }| \text{ } \exists \text{ } j' \rightarrow i \in R_1 \}=\emptyset.$$ It is called \hdef{red} if $$\{j'\in Q_0'\text{ }| \text{ } \exists \text{ } j' \leftarrow i \in R_1 \}=\emptyset.$$ 
\end{definition}

While it is not clear from the definition that every mutable vertex in $R_0$ must be
either red or green, this was shown to be true for quivers in~\cite{derksen} and
was also shown to be true in a more general setting in~\cite{gross}.

\begin{theorem}\cite{derksen,gross}\label{thm:signcoh}
Let $R \in \mut(\hat{Q})$. Then every mutable vertex in $R_0$ is either red or green.
\end{theorem}

\begin{definition}[{\cite[Definition 2.8]{brustle}, \cite[Section 5.14]{keller2}}]
A \hdef{green sequence} for $Q$ is a sequence of vertices $\mgsi=(i_1, \dotsc, i_l)$ of $Q$ such that $i_1$ is green in $\hat{Q}$ and for any $2\leq k \leq l$, the vertex $i_k$ is green in $\mu_{i_{k-1}}\circ \dotsb \circ \mu_{i_1}(\hat{Q})$.
%The integer $l$ is called the length of the sequence $\mgsi$ and is denoted by $l(\mgsi)$. 
A green sequence $\mgsi$ is \hdef{maximal} if every mutable vertex in $\mu_{i_{l}}\circ \dotsb \circ \mu_{i_1}(\hat{Q})$ is red.
\end{definition} 

In~\cite{brustle}, Br\"ustle, Dupont and P\'erotin showed that a maximal green
sequence preserves the quiver:

\begin{lemma}\cite[Proposition 2.10]{brustle}\label{lem:iso_mgs}
	If $\mgsi$ is a maximal green sequence for a quiver $Q$ then $\mu_{\mgsi}(Q)$
	is isomorphic to $Q$.
\end{lemma}

Therefore there is some permutation $\sigma$ on the vertices of the quiver which
maps $\mu_{\mgsi} \left( Q \right)$ to $Q$. This permutation is referred to as
the permutation induced by $\mgsi$.

\subsection{Existence of maximal green sequences}
A number of results have been proved determining when quivers can
admit a maximal green sequence.

\begin{lemma}\cite[Proposition 2.5]{brustle}\label{lem:opposite_quiver}
	A quiver $Q$ has a maximal green sequence if and only if its opposite quiver
	$Q^\text{op}$ has a maximal green sequence.
\end{lemma}

We call a quiver $Q$ \hdef{acyclic} if there are no oriented cycles in $Q$.
Any such quiver contains at least one source vertex, which is not the target of
any arrows.

\begin{theorem}\cite[Lemma 2.20]{brustle}\label{thm:acyclic}
If a quiver is acyclic, then it has a maximal green sequence.
In particular, a maximal green sequence can always be found by repeatedly mutating at sources.
\end{theorem}

In~\cite{muller}, Muller uses the relationship between maximal green sequences
and paths in scattering diagrams to study green sequences of induced subquivers.

\begin{theorem}{\cite[Theorem 1.4.1]{muller}}\label{thm:subquiver}
If $Q$ admits a maximal green sequence, then any induced subquiver has a maximal green sequence. 
\end{theorem}

Using this one can prove a quiver does not admit a maximal green sequence by
finding an induced subquiver which does not admit a maximal green sequence.
However there are examples of quivers which do not have a maximal green sequence
whereas every induced subquiver does.

At the same time, Muller shows that the existence of maximal green sequences is
not an invariant under mutation~\cite[Section 2]{muller}, so one quiver in a
mutation class could admit a maximal green sequence while another might not.
The second author shows in~\cite{mills} that in the case of mutation-finite quivers this
does not occur, and in fact only very specific mutation classes of
quivers do not admit maximal green sequences.

\begin{theorem}\cite[Theorem 3]{mills}\label{thm:mu_finite_mgs}
Let $Q$ be a mutation-finite quiver. $Q$ has no maximal green sequence if and only if $Q$ arises from a triangulation of a once-punctured closed surface or is one of the two quivers in the mutation class of $\mathbb{X}_7$.
\end{theorem}

In~\cite{brustle2}, Br\"ustle, Hermes, Igusa and Todorov build a series of
results showing how c-vectors and c-matrices change as mutations are applied
along a maximal green sequence. These results allow the authors to construct a
maximal green sequence for any quiver which appears as an intermediary of the
initial maximal green sequence.

\begin{theorem}[{Rotation Lemma~\cite[Theorem 3]{brustle2}}]\label{thm:rotationlem}
	If $\mgsi = \left( i_1, i_2, \dotsc, i_\ell \right)$ is a maximal green
	sequence of a quiver $Q$ with induced permutation $\sigma$, then the sequence
	$\left( i_2, \dotsc, i_\ell, \sigma^{-1}(i_1) \right)$ is a maximal green
	sequence for the quiver $\mu_{i_1}(Q)$ with the same induced permutation
	$\sigma$.
\end{theorem}

A maximal green sequence is a cycle in the quiver exchange graph, and the
rotation lemma shows that this cycle always gives a maximal green sequence for
any quiver appearing in that cycle. The statement above rotates in only one
direction, however this direction can be reversed:

\begin{theorem}[{Reverse Rotation Lemma}]\label{thm:reverserotationlem}
	If $\mgsi = \left( i_1, i_2, \dotsc, i_{\ell-1}, i_\ell \right)$ is a maximal green
	sequence of a quiver $Q$ with induced permutation $\sigma$, then the sequence
	$\left( \sigma(i_\ell), i_1, i_2, \dotsc, i_{\ell-1} \right)$ is a maximal green
	sequence for the quiver $\mu_{\sigma(i_\ell)}(Q)$ with the same induced
	permutation $\sigma$.
\end{theorem}

\begin{proof}
	Applying the rotation lemma $\ell-1$ times gives that $\big( i_\ell,
	\sigma^{-1}(i_1), \dotsc, \sigma^{-1}(i_{\ell-1}) \big)$ is a maximal green
	sequence for $\mu_{i_{\ell-1}} \dotsb \mu_{i_2} \mu_{i_1} (Q)$. As $\sigma$ is
	the induced permutation for $\mgsi$, we have $\sigma\big(\mu_{\mgsi}(Q)\big) =
	Q$, so $\mu_{\mgsi}(Q) = \mu_{i_\ell} \mu_{i_{\ell-1}} \dotsb \mu_{i_1}(Q) =
	\sigma^{-1}(Q)$ and therefore
	\[ \mu_{i_{\ell-1}} \dotsb \mu_{i_1}(Q) = \mu_{i_\ell}\left( \sigma^{-1}(Q)
	\right) = \sigma^{-1}\left( \mu_{\sigma(i_\ell)} (Q) \right). \]
	Hence $\big( i_\ell, \sigma^{-1}(i_1), \dotsc, \sigma^{-1}(i_{\ell-1}) \big)$
	is a maximal green sequence for $\sigma^{-1}\left( \mu_{\sigma(i_\ell)} (Q)
	\right)$ and by applying $\sigma$ to both the sequence of vertices and to the
	quiver, we get that $\big( \sigma(i_\ell), i_1, i_2, \dotsc, i_{\ell-1} \big)$
	is a maximal green sequence for $\mu_{\sigma(i_\ell)}(Q)$.
\end{proof}

\subsection{Direct sums of quivers}
Garver and Musiker showed in~\cite{garver} that if a quiver $Q$ can be written as a direct sum
of quivers, where each summand has a maximal green sequence, then $Q$ has a maximal green
sequence itself. Throughout this subsection we assume that $(Q,F)$ and $(Q',F')$ are finite ice quivers with vertices labelled $Q_0\setminus F=\lbrace 1,\dots,N_1 \rbrace$ and $Q'_0 \setminus F' = \lbrace N_1+1,\dots,N_1+N_2\rbrace$. 

\begin{definition}\cite[Definition 3.1]{garver}
Let $(a_1,\dots,a_k)$ denote a $k$-tuple of elements from $Q_0\setminus F$ and $(b_1,\dots,b_k)$ denote a $k$-tuple of elements from $Q'_0\setminus F'$. Then for any ice quivers $(R,F) \in \mut((Q,F))$ and $(R',F') \in \mut((Q',F'))$ we define the \hdef{direct sum} of $(R,F)$ and $(R',F')$, denoted $(R,F) \oplus_{(a_1,\dots,a_k)}^{(b_1,\dots,b_k)}(R',F')$, to be the ice quiver with vertices
$$\big((R,F) \oplus_{(a_1,\dots,a_k)}^{(b_1,\dots,b_k)}(R',F')\big)_0= R_0 \sqcup R'_0$$
and edges
$$\big((R,F) \oplus_{(a_1,\dots,a_k)}^{(b_1,\dots,b_k)}(R',F')\big)_1= (R,F)_1 \sqcup (R',F')_1 \sqcup \{ a_i \rightarrow b_i | i = 1,\dots, k\}.$$
We say that $(R,F) \oplus_{(a_1,\dots,a_k)}^{(b_1,\dots,b_k)}(R',F')$ is a \hdef{$t$-colored direct sum} if $t= \#\{$ distinct elements of $\{a_1,\dots,a_k\}\}$ and there does not exist $i$ and $j$ such that $$\#\{a_i \rightarrow b_j\} \geq 2.$$
\end{definition}

\begin{theorem}\cite[Theorem 3.12]{garver}\label{thm:directsum-mgs}
Let $Q=Q'\oplus_{(a_1,\dots,a_k)}^{(b_1,\dots,b_k)} Q''$ be a $t$-colored
direct sum of quivers. If $\left( i_1, \dotsc, i_r \right)$ is a maximal green
sequences for $Q'$ and $\left( j_1, \dotsc, j_s \right)$ is a maximal green
sequence for $Q''$, then $\left( i_1, \dotsc, i_r, j_1, \dotsc, j_s \right)$
is a maximal green sequence for $Q$.
\end{theorem}

\subsection{Quivers ending with k-cycles}
Direct sums give a method to construct maximal green sequences by decomposing a
quiver into disjoint induced subquivers. Another approach involves considering
oriented k-cycles appearing in the quiver.

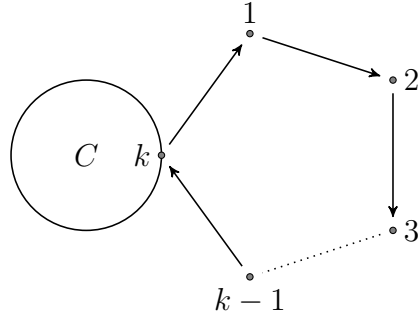
\begin{figure}
\begin{tikzpicture}
	\tkzDefPoint(0,0){k}
	\tkzDefShiftPoint[k](54:2){a}
	\tkzDefShiftPoint[a](-18:2){b}
	\tkzDefShiftPoint[b](-90:2){c}
	\tkzDefShiftPoint[k](-54:2){k1}
	\node (label) at (-1, 0) {$C$};
	\tkzDrawCircle(label,k)
	\tkzDrawPoints(a,b,c,k1,k)
	\tkzLabelPoint[above](a){$1$}
	\tkzLabelPoint[right](b){$2$}
	\tkzLabelPoint[right](c){$3$}
	\tkzLabelPoint[below](k1){$k-1$}
	\tkzLabelPoint[left](k){$k$}
	\tkzDrawSegments[->, shorten <=5, shorten >=5, >=stealth'](k,a a,b b,c k1,k)
	\tkzDrawSegments[dotted, shorten <=5, shorten >=5](c,k1)
\end{tikzpicture}
\caption{The decomposition of a quiver $Q$ into a k-cycle, with vertices
	$1,\dotsc k$, and an induced subquiver $C$, as described in
	Theorem~\ref{thm:triangleend}.}
\label{fig:kcycle}
\end{figure}

\begin{theorem}\label{thm:triangleend}
	Let $Q$ be a quiver containing an oriented k-cycle with vertices labelled $1,
	\dotsc, k$, where each vertex $1, \dotsc, k-1$ is only adjacent to two arrows,
	one from the vertex preceding it in the cycle and one to the next vertex in
	the cycle, as shown in Figure~\ref{fig:kcycle}. Let $C = Q\left[ Q_0 \setminus
	\lbrace 1,\dotsc,k-1 \rbrace \right]$ be the induced subquiver obtained by removing
	the vertices in the k-cycle, excluding vertex $k$. Then $Q$ has a maximal
	green sequence if and only if $C$ has a
	maximal green sequence.

	In particular, if $Q$ contains a 3-cycle in this form, such that if
	$\mgsi_C$ is a maximal green sequence for $C$, then $\left( 2,
	\mgsi_C, 1, 2 \right)$ is a maximal green sequence for $Q$.
\end{theorem}

\begin{proof}
	Clearly if the subquiver $C$ does not admit a maximal green sequence then by
	Theorem~\ref{thm:subquiver} the full quiver $Q$ cannot either, so assume that
	$C$ does admit a maximal green sequence.

	The mutated quiver $Q' = \mu_2 \circ \mu_3 \circ \dotsb \circ \mu_{k-1}(Q)$
	can be decomposed into the direct sum of $C$ and a quiver of type $A_{k-1}$:
	\[Q' = C \oplus^{2}_{k} \left( 1 \leftarrow 2 \to 3 \to \dotsb \to k-1
	\right). \]
	The type $A_{k-1}$ quiver has a maximal green sequence $\left( 1, \dotsc, k-1,
	1, k-1, k-2, \dotsc, 3 \right)$ with induced permutation $\sigma_A = \left(
	\;(k-1)\enspace (k-2)\enspace \dotsb\enspace 2\enspace 1\; \right)$ which
	shifts the index of each vertex down one, modulo $k-1$. By assumption $C$ has
	maximal green sequence $\mgsi_C$ with induced permutation $\sigma_C$, so
	Theorem~\ref{thm:directsum-mgs} shows that $\left( \mgsi_C, 1, \dotsc,
	k-1, 1, k-1, \dotsc, 3 \right)$ is a maximal green sequence for $Q'$ with
	induced permutation $\sigma = \sigma_C \circ \sigma_A$. As $1, \dotsc k-1$ are
	not in the vertex set of $C$, its restriction $\sigma\big|_{1,\dotsc,k-1} =
	\sigma_A$ so by the reverse rotation lemma,
	Theorem~\ref{thm:reverserotationlem},
	\[\left( k-1, \dotsc, 2, \mgsi_C, 1, 2, \dotsc, k-1 \right)\]
	is a maximal green sequence for $\mu_{k-1} \dots \mu_2 \left( \mu_2 \dotsc
	\mu_{k-1}(Q) \right) = Q$.

	In the case when $k=3$ then the maximal green sequence is $\left( 2,
	\mgsi_C, 1, 2 \right)$ as claimed.
\end{proof}

To allow us to easily refer to quivers of the form required in
Theorem~\ref{thm:triangleend} and illustrated in Figure~\ref{fig:kcycle} we say
such quivers \hdef{end in a k-cycle}.

\section{Cluster Algebras and Upper Cluster Algebras}\label{sec:upper}
Let $m,n$ be positive integers such that $m\geq n$. Denote $\mathcal{F}=\mathbb{Q}(x_1,\dots,x_m)$.
A \hdef{seed} $\Sigma=(\tilde{\bf x},\tilde{Q})$ is a pair where $\tilde{\bf x}=\{x_1,\ldots,x_m\}$ is an $m$-tuple of elements of $\mathcal{F}$ that form a free generating set, and $\tilde{Q}$ is an ice quiver with $n$ mutable vertices and $m-n$ frozen vertices. 

%For any integer $a$, let $[a]_+:= \max(0,a)$.
For a seed $(\tilde{\bf x},\tilde{Q})$ and a specified index $1\leq k \leq n$, define the \hdef{seed mutation} of $(\tilde{\bf x},\tilde{Q})$ at $k$, denoted $\mu_{k}(\tilde{\bf x},\tilde{Q})$, to be a new seed $(\tilde{\bf x}',\tilde{Q}')$ where $\tilde{Q}'$ is the quiver $\mu_k(\tilde{Q})$ defined in Definition~\ref{def:mutation} and $\tilde{\bf x}'=\{x_1',\ldots,x_m'\}$ with
$$
\aligned
 x_j'&= \begin{cases} 
\displaystyle{x_k^{-1}\left(\prod_{i \leftarrow k \in Q'}x_i+\prod_{i \rightarrow k \in Q'}x_i\right)}
  & \mbox{if } j=k; \\
x_j & \mbox{otherwise.} \end{cases}
\endaligned
$$
Note that seed mutation is an involution, so mutating $(\tilde{\bf x}', \tilde{Q}')$ at $k$ will return to our original seed $(\tilde{\bf x}, \tilde{Q})$. 

Two seeds $\Sigma_1$ and $\Sigma_2$ are said to be \hdef{mutation-equivalent} or in the same \hdef{mutation class} if $\Sigma_2$ can be obtained by a sequence of mutations from $\Sigma_1$. This is obviously an equivalence relation.

%
% These few paragraphs were the only time that y variables/coefficients are
% used.
In a given seed $(\tilde{\bf x}, \tilde{Q})$, we call the set ${\bf
x}=\{x_1,\dots,x_n\}$ the \hdef{cluster} of the seed and each element of a
cluster the \hdef{cluster variables}.  This emphasizes the different roles
played by $x_i \; (i \le n)$ and $x_i\; (i>n)$, where those $x_i$ with $i > n$
are linked to coefficients of the cluster algebra and we denote
\[\mathbb{ZP}=\mathbb{Z}[x_{n+1}^{\pm1},\dots,x_m^{\pm1}].\]

In the paper, we shall only study cluster algebras of geometric type, defined as follows.
 
\begin{definition}
	Given a seed $(\tilde{\bf x},\tilde{Q})$, the \hdef{cluster algebra}
	$\mathcal{A}(\tilde{\bf x},\tilde{Q})$ \hdef{of geometric type} is the subring
	of $\mathcal{F}$ generated over $\mathbb{ZP}$ by all cluster variables
	appearing in all seeds that are mutation-equivalent to $(\tilde{\bf
	x},\tilde{Q})$. The seed $(\tilde{\bf x},\tilde{Q})$ is called the
	\hdef{initial seed} of $\mathcal{A}(\tilde{\bf x},\tilde{Q})$. 
\end{definition}

It follows from the definition that any seed in the same mutation class will generate the same cluster algebra up to isomorphism. 
 
\begin{definition}
Given a cluster algebra $\mathcal{A}$, the \hdef{upper cluster algebra} $\mathcal{U}$ is defined as $$ \mathcal{U} =
 \bigcap_{{\bf x}=\{x_1,\ldots,x_n\}}
  \mathbb{ZP}[x_1^{\pm 1},\ldots,x_n^{\pm 1}]$$ where ${\bf x}$ runs over all clusters of $\mathcal{A}$.
\end{definition}
By the Laurent Phenomenon \cite[Theorem 3.1]{cluster1} there is a natural containment $\mathcal{A} \subseteq \mathcal{U}$.

If a quiver $Q$ is mutation-equivalent to an acyclic quiver we say that $Q$ is \hdef{mutation-acyclic.} We say that the cluster algebra $\mathcal{A}(\tilde{\bf x},\tilde{Q})$ is \hdef{acyclic} if $Q$ is mutation-acyclic; otherwise we say that the cluster algebra is \hdef{non-acyclic}.

\begin{theorem}{\cite{BFZ,M}}\label{thm:acyclic_AU}
If $\mathcal{A}$ is an acyclic cluster algebra, then $\mathcal{A}=\mathcal{U}$. 
\end{theorem}

\subsection{Locally acyclic cluster algebras and the Louise property}
Muller expanded this result to \hdef{locally acyclic} cluster algebras, which are those that can be covered by finitely many acyclic cluster localizations. We refer the reader to the paper \cite{M2} for a full definition, but instead we recall a sufficient property, called the Louise property, which implies local acyclicity of the cluster algebra as well as many other nice properties. 

\begin{definition}
Let $(Q,F)$ be an ice quiver. We define $i \rightarrow j \in Q_1$ to be a \hdef{separating edge} of $Q$ if $i$ and $j$ are mutable and there is no bi-infinite path through the edge $i \rightarrow j$. 
\end{definition}
\begin{definition}
Let $V \subset Q_0$ and let $Q[{V}]$ denote the induced subquiver of $Q$ with vertex set $V$. 
The Louise property is then defined recursively. We say that a quiver satisfies the Louise property, or is Louise, if either \begin{enumerate}
\item $Q$ has no edges, 
\item or there exists $Q' \in \mut(Q)$ that has a separating edge $i \rightarrow j$, such that the quivers $Q'[{Q'_0 \setminus \{i\}}],Q'[{Q'_0 \setminus \{j\}}]$, and $Q'[{Q'_0 \setminus \{i,j\}}]$ are Louise. 
\end{enumerate}
A cluster algebra is Louise if it is generated by a Louise quiver. 
\end{definition}
In particular any acyclic quiver has the Louise property. 
\begin{theorem}\cite[Proposition 2.6]{MS} 
If a cluster algebra is Louise, then it is locally acyclic. 
\end{theorem}

\begin{theorem}\cite[Theorem 2]{M}\label{thm:LA_AU}
If $\mathcal{A}$ is locally acyclic, then $\mathcal{A}=\mathcal{U}$. 
\end{theorem}

\section{On the mutation class of minimal mutation-infinite quivers}\label{sec:mmi-mut}
Minimal mutation-infinite quivers were first studied by Felikson, Shapiro and
Tumarkin in~\cite[Section 7]{felikson}, where such quivers were shown to exist
only up to rank 10. All such quivers were later classified in~\cite{lawson} into
a number of move-classes, each of which has a distinguished representative as
either an orientation of a hyperbolic Coxeter simplex diagram, a double arrow
quiver or one of the exceptional quivers. 
\begin{definition}[{\cite[Section 4]{lawson}}]
	A \hdef{minimal mutation-infinite move} is a sequence of mutations of a
	minimal mutation-infinite quiver which preserves the property of being minimal
	mutation-infinite. As such the image of any minimal mutation-infinite quiver
	under a move is another minimal mutation-infinite quiver.
\end{definition}

These moves then classify all minimal mutation-infinite quivers under the
equivalence relation where two minimal mutation-infinite quivers are
move-equivalent if there is a sequence of moves taking one quiver to the other. 

\begin{definition}
	A \hdef{move-class} of a minimal mutation-infinite quiver is the equivalence
	class under this relation containing that quiver.
\end{definition}
All minimal mutation-infinite quivers belong to one of 47 move-classes, which
have representatives given in Figures~\ref{fig:hcsreps},~\ref{fig:dareps}
and~\ref{fig:excreps} as described in~\cite[Section 5]{lawson}. 
\begin{figure}
\includegraphics{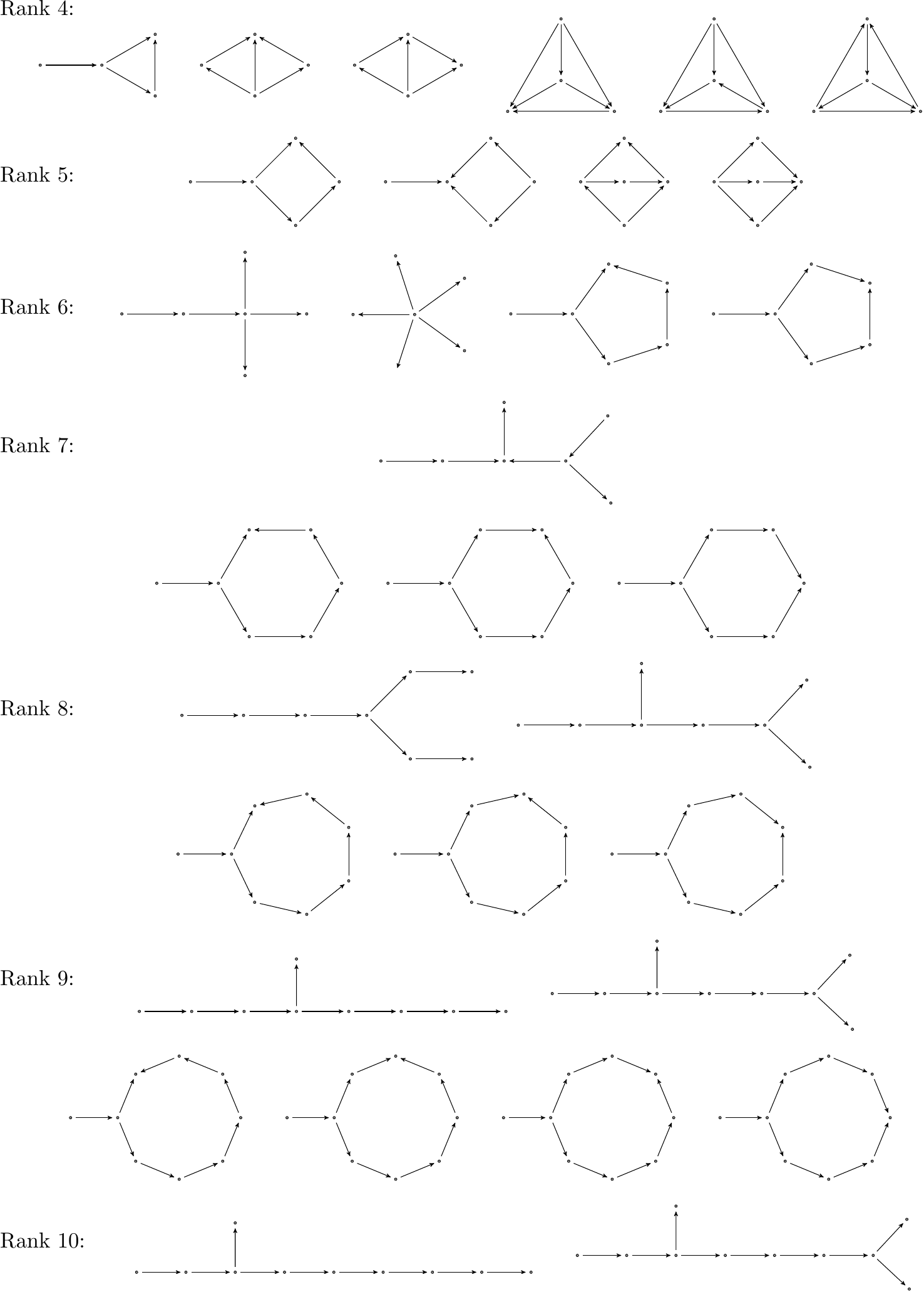}
\caption{Representatives of minimal mutation-infinite classes arising from
orientations of hyperbolic Coxeter simplex diagrams as given in~\cite{lawson}.}
\label{fig:hcsreps}
\end{figure}
\begin{figure}
\includegraphics{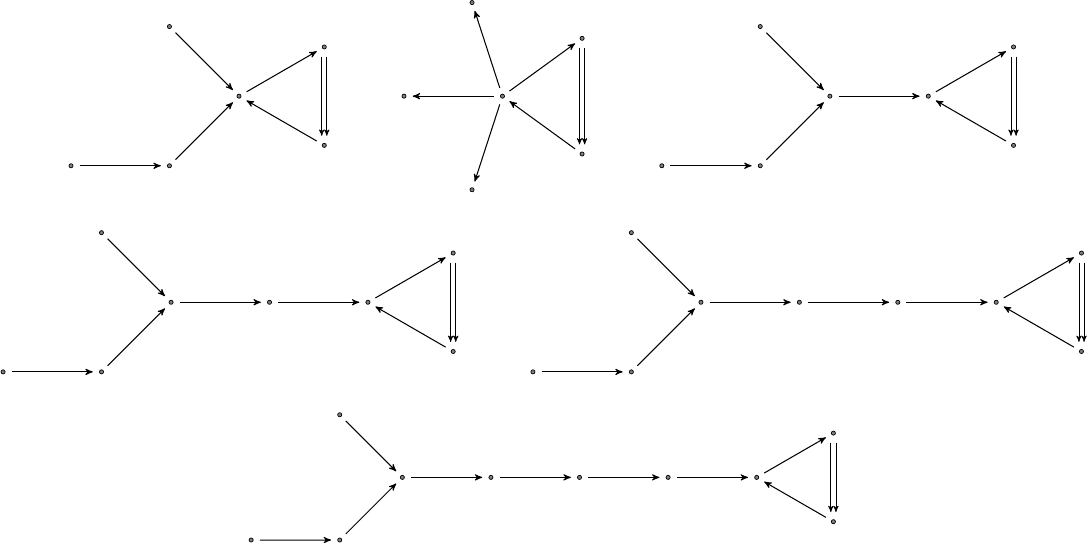}
\caption{Double arrow representatives of minimal mutation-infinite
classes as given in~\cite{lawson}.}
\label{fig:dareps}
\end{figure}
\begin{figure}
\includegraphics{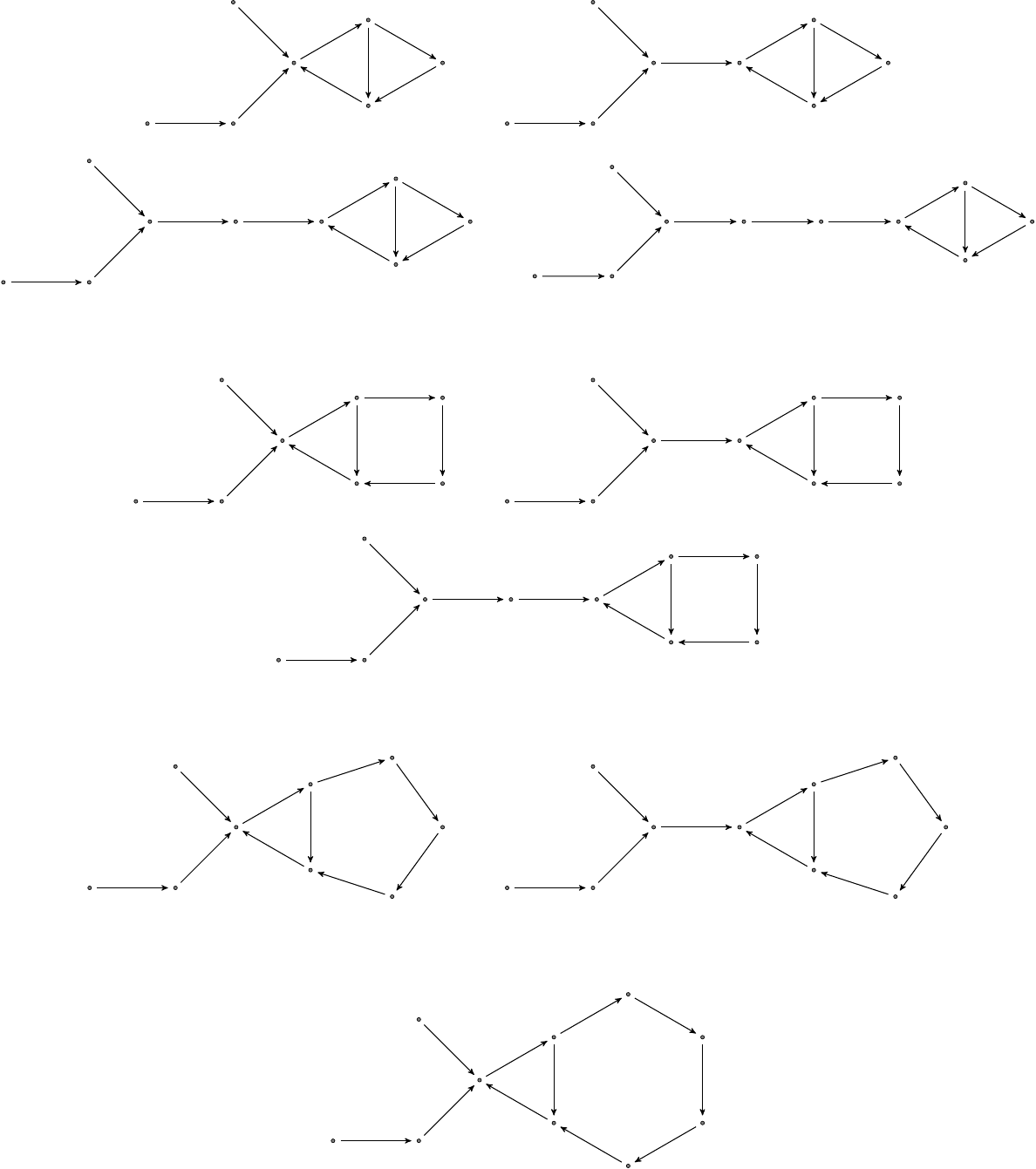}
\caption{Exceptional representatives of minimal mutation-infinite
classes as given in~\cite{lawson}.}
\label{fig:excreps}
\end{figure}

In this section we show which move-classes are mutation-acyclic and which are not mutation-acyclic. We also use this information together with another mutation invariant to show that many of the move-classes generate distinct mutation classes. 

\subsection{Hyperbolic Coxeter simplex representatives.}
We now determine if distinct move-classes of minimal mutation-infinite quivers with hyperbolic Coxeter representatives belong to the same mutation class.

For a rank $n$ quiver $Q$ we associate an $n \times n$ matrix $B_Q=(b_{ij})$, where $b_{ij} = 0$ if $i=j$ and otherwise $b_{ij}$ is the number of edges $i \rightarrow j \in Q_1$.  Note that if $i \rightarrow j \in Q_1$ then $b_{ji} < 0$. 

It was shown in \cite[Lemma 3.2]{BFZ} that for a quiver $Q$ the rank of the matrix $B_Q$ is preserved by mutation and hence is an invariant of the mutation class of $Q$.
% It is a well-known fact that the determinant of the matrix $B_Q$ is also invariant under mutation.
We use the rank of the matrix $B_Q$ and the number of acyclic quivers in the mutation class to distinguish the mutation classes of minimal mutation-infinite quivers. 
%These two invariants are not enough to distinguish the different move-classes of minimal mutation-infinite quivers so we also consider the number of acyclic quivers in the mutation class. 

\begin{theorem}\cite[Corollary 4]{caldero-keller}
The set of acyclic quivers form a connected subgraph of the exchange graph. Furthermore, if $Q$ is an acyclic quiver, then every acyclic quiver in $\mut(Q)$ can be obtained from $Q$ by a sequence of sink/source mutations. 
\end{theorem}
\begin{corollary}
There are finitely many acyclic quivers in any mutation class. 
\end{corollary}
\begin{proof}
Clearly a sink/source mutation preserves the underlying graph of a quiver and there are only finitely many orientations of this graph. Therefore the number of these orientations that are acyclic is also finite.
\end{proof}

\begin{theorem}\label{thm:hcs_diff_class}
Mutating the sixth rank 4 quiver in Figure~\ref{fig:hcsreps} at the top vertex yields the first rank 4 quiver, so both of their move-classes are in the same mutation class. For the rest of the move-classes of minimal mutation-infinite quivers that are represented by an orientation of a hyperbolic Coxeter simplex diagram, each move-class determines a distinct mutation class.
\end{theorem}
\begin{proof}
The result follows from comparing the $B$-matrix rank, quiver rank, and number of acyclic seeds in the mutation class. The values of these statistics are given in Table~\ref{tab1} and Table~\ref{tab2} in the appendix. The fourth rank 4 quiver is shown to be not mutation-acyclic in Theorem~\ref{thm:rank4_not_acyclic}.
\end{proof}
\subsection{Admissible quasi-Cartan companion matrices}
While all but one of the Coxeter simplex move-classes are mutation-acyclic, the
double arrow move-classes are not, and therefore these move-classes cannot
belong to the same mutation classes as those of the Coxeter simplex representatives. The proof that a quiver is not
mutation-acyclic relies heavily on the idea of admissible quasi-Cartan
companions introduced by Seven in~\cite{seven-semipos} building on work by
Barot, Geiss and Zelevinsky~\cite{BGZ-cartan}.

\begin{definition}[{\cite[Section 1]{BGZ-cartan}}]
	Let $B = \left( b_{i,j} \right)$ be a skew-symmetric matrix representing a
	quiver $Q$ of rank $n$. A \hdef{quasi-Cartan companion} of $Q$ is a
	symmetric matrix $A = \left( a_{i,j} \right)$ such that $a_{i,i} = 2$ for $i =
	1, \dotsc, n$ and $| a_{i,j} | = | b_{i,j} |$ for $i \neq j$.
\end{definition}

\begin{definition}[{\cite[Section 1]{seven-symmetric}}]
	A \hdef{cycle} $Z$ in a quiver $Q$ is an induced subquiver of $Q$ whose
	vertices can be labelled by elements of $\mathbb{Z}/k\mathbb{Z}$ so that the
	only edges appearing in $Z$ are $\lbrace i, i+1 \rbrace$ for $i \in
	\mathbb{Z}/k\mathbb{Z}$.
\end{definition}

\begin{definition}[{\cite[Definition 2.10]{seven-semipos}}]
	A quasi-Cartan companion $A = \left( a_{i,j} \right)$ of a quiver $Q$ is
	\hdef{admissible} if for any cycle $Z$ in $Q$, if $Z$ is an oriented
	(respectively, non-oriented) cycle then there are an odd (resp., even) number
	of edges $\lbrace i,j \rbrace$ in $Z$ such that $a_{i,j} > 0$.
\end{definition}

A quasi-Cartan companion of a quiver can be thought of as assigning signs
(either~$+$ or~$-$) to the edges of a quiver, determined by whether $a_{i,j} > 0
$ or $< 0$. This labelling of the edges is admissible if the number of $+$'s in
an oriented (resp., non-oriented) cycle is odd (even).

\begin{lemma}[{\cite[Theorem 2.11]{seven-semipos}}]\label{lem:adm-flip}
	If $A$ and $A'$ are two admissible quasi-Cartan companions of a quiver $Q$,
	then they can be obtained from one another by a sequence of simultaneous sign
	changes in rows and columns.
\end{lemma}

As a row and column in the matrix corresponds to a vertex in the quiver, this
simultaneous sign change at row and column $k$ is equivalent to flipping the
signs on all edges in the quiver adjacent to vertex $k$.

\begin{theorem}[{\cite[Theorem 1.2]{seven-symmetric}}]\label{thm:adm-mutacyc}
	If $Q$ is a mutation-acyclic quiver, then $Q$ has an admissible quasi-Cartan
	companion.
\end{theorem}

We are now ready to show that the fourth rank 4 minimal mutation-infinite quiver is not mutation-acyclic. 

\begin{figure}
\begin{tikzpicture}
\tkzDefPoint(0,0){A}
\tkzDefShiftPoint[A](-30:2.5){B}
\tkzDefShiftPoint[A](90:2.5){C}
\tkzDefShiftPoint[A](210:2.5){D}
\tkzDrawPoints(A,B,C,D)
\tkzLabelPoint[right](B){$1$}
\tkzLabelPoint[below](A){$2$}
\tkzLabelPoint[left](C){$3$}
\tkzLabelPoint[left](D){$4$}
\tkzLabelSegment[above](A,B){$a$};
\tkzLabelSegment[right](A,C){$b$};
\tkzLabelSegment[above](A,D){$c$};
\tkzLabelSegment[above right](B,C){$d$};
\tkzLabelSegment[above](B,D){$e$};
\tkzLabelSegment[above left](D,C){$f$};
\tkzDrawSegments[->, shorten >=5, shorten <=5, >=stealth'](B,D B,A B,C A,D C,A D,C)
\end{tikzpicture}
\caption{Minimal mutation-infinite representative not mutation-equivalent to
any acyclic quiver.}
\label{fig:mut-acyclic-rep}
\end{figure}
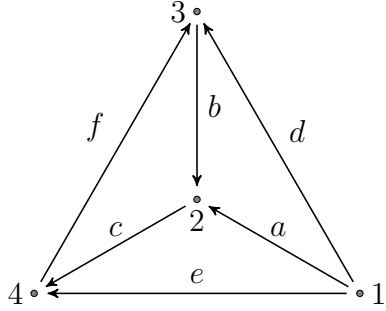
\begin{theorem}\label{thm:rank4_not_acyclic}
The minimal-mutation infinite quiver $Q$ depicted in Figure~\ref{fig:mut-acyclic-rep} does not have any admissible quasi-Cartan companion, and is hence not mutation-acyclic.
\end{theorem}
\begin{proof}
	Assume the edges and vertices of $Q$ are labelled as in
	Figure~\ref{fig:mut-acyclic-rep}, and that the quiver has an admissible
	quasi-Cartan companion. We now use Lemma~\ref{lem:adm-flip} to canonically
	label the edges of $Q$, up to flipping the signs at vertices.
	
	The triangle $(2,3,4)$ in $Q$ is oriented, so must have an odd number of edges
	labelled~$+$, in particular it must have at least one~$+$. By flipping
	vertices $3$ and $4$ we can ensure that the label of $f$ is~$+$.  Then, either
	$b$ and $c$ are both~$+$ or are both~$-$, and by flipping $2$ we can choose
	them to be~$+$. So far we have $a,b$ and $c$ are all~$+$ and we have fixed or
	flipped vertices $1, 2$ and $5$.

	Now the two non-oriented triangles $(1, 2, 4)$ and $(1, 2, 3)$ require an even
	number of~$+$'s, while they each already have at least one. Hence either $d$ and $e$ are both~$+$ and $a$ is~$-$ or $d$ and $e$ are both~$-$ and $a$ is~$+$. By flipping
	1 we can choose the former. We have now fixed or
	flipped all vertices and have assigned the only possible labels to all the
	edges, up to flipping the signs at vertices.

	However we now have a non-oriented cycle $(1,3,4)$ which contains three~$+$'s, 				which gives a contradiction as a non-oriented cycle must have an even
	number of positive edges in an admissible companion. Hence $Q$ cannot admit an
	admissible quasi-Cartan companion and so by Theorem~\ref{thm:adm-mutacyc}
	cannot be mutation-acyclic.
\end{proof}
\subsection{Double arrow minimal mutation-infinite quivers.}
 We use a similar approach to show that the move-classes with double arrow representatives are not mutation-acyclic. It then follows that the mutation classes of double arrow representatives are distinct from the mutation classes of the minimal mutation-infinite quivers with hyperbolic Coxeter representatives. 
\begin{lemma}[{\cite[Corollary 5.3]{BMR-clustermut}}]\label{lem:sub-acyc}
	If $Q$ is mutation-acyclic, then any induced subquiver of $Q$ is also
	mutation-acyclic.
\end{lemma}

\begin{figure}
	\begin{tikzpicture}[>=stealth',shorten <=5,shorten >=5]
		\def\edgesize{4}
		\tkzDefPoint(0,0){A};
		\tkzDefPoint(\edgesize,0){B};
		\tkzDefPoint(\edgesize,\edgesize){C};
		\tkzDefPoint(0,\edgesize){D};
		\tkzDefPoint(.5*\edgesize,.5*\edgesize){E};
		\tkzFillPolygon[color=gray!15](A,D,E);
		\tkzFillPolygon[color=gray!15](B,C,E);
		\tkzDrawPoints(A,B,C,D,E)
		\tkzLabelPoint[above left](D){$1$};
		\tkzLabelPoint[above right](C){$2$};
		\tkzLabelPoint[below right](B){$3$};
		\tkzLabelPoint[below left](A){$4$};
		\tkzLabelPoint[right=.2 of E](E){$5$};
		\tkzDrawSegments[->](A,B B,C C,D D,A D,E E,A B,E E,C)
		\tkzLabelSegment[above](D,C){$a$};
		\tkzLabelSegment[above](D,E){$b$};
		\tkzLabelSegment[above](E,C){$c$};
		\tkzLabelSegment[left](D,A){$d$};
		\tkzLabelSegment[right](C,B){$e$};
		\tkzLabelSegment[below](E,A){$f$};
		\tkzLabelSegment[below](E,B){$g$};
		\tkzLabelSegment[below](A,B){$h$};
	\end{tikzpicture}
	\caption{Quiver which does not have an admissible quasi-Cartan companion. The
	vertices are labelled $1, \dotsc, 5$ and the edges $a, \dotsc, h$.
	Non-oriented cycles are shaded gray.}
	\label{fig:non-admissible}
\end{figure}
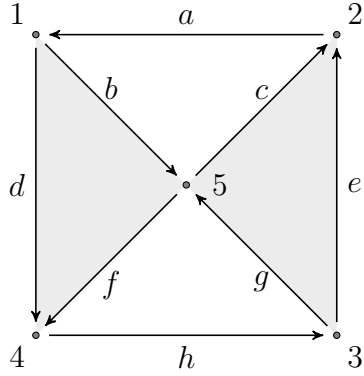

\begin{lemma}\label{lem:adm-nonacyc}
	The quiver $Q$ depicted in Figure~\ref{fig:non-admissible} does not have any
	admissible quasi-Cartan companion, and is hence not mutation-acyclic.
\end{lemma}

\begin{proof}
	Assume the edges and vertices of $Q$ are labelled as in
	Figure~\ref{fig:non-admissible}, and that the quiver has an admissible
	quasi-Cartan companion. We now use Lemma~\ref{lem:adm-flip} to canonically
	label the edges of $Q$, up to flipping the signs at vertices.
	
	The triangle $(1,5,2)$ in $Q$ is oriented, so must have an odd number of edges
	labelled~$+$, in particular it must have at least one~$+$. By flipping
	vertices $1$ and $2$ we can ensure that the label of $a$ is~$+$.  Then, either
	$b$ and $c$ are both~$+$ or are both~$-$, and by flipping $5$ we can choose
	them to be~$+$. So far we have $a,b$ and $c$ are all~$+$ and we have fixed or
	flipped vertices $1, 2$ and $5$.

	Now the two non-oriented triangles $(1, 4, 5)$ and $(2, 3, 5)$ require an even
	number of~$+$'s, while they each already have at least one. Hence one of $d$
	or $f$ is~$+$ and the other is~$-$, and the same for $e$ and $g$. By flipping
	$4$ and $3$ we can choose that $f$ and $g$ are~$+$ while $d$ and $e$ are~$-$.
	This leaves the oriented triangle $(3, 5, 4)$ with two~$+$'s, but it requires
	an odd number of positives, so $h$ must also be $+$. We have now fixed or
	flipped all vertices and have assigned the only possible labels to all the
	edges, up to flipping the signs at vertices.

	However we now have an oriented cycle $(1,2,3,4)$ which contains two~$+$'s and
	two~$-$'s, which gives a contradiction as an oriented cycle must have an odd
	number of positive edges in an admissible companion. Hence $Q$ cannot admit an
	admissible quasi-Cartan companion and so by Theorem~\ref{thm:adm-mutacyc}
	cannot be mutation-acyclic.
\end{proof}

\begin{lemma}\label{lem:da-mutacyc}
	Each double arrow representative is mutation-equivalent to a quiver containing
	the quiver depicted in Figure~\ref{fig:non-admissible} as an induced
	subquiver. Hence each double arrow move-class is not mutation-acyclic.
\end{lemma}

\begin{figure}
	\includegraphics{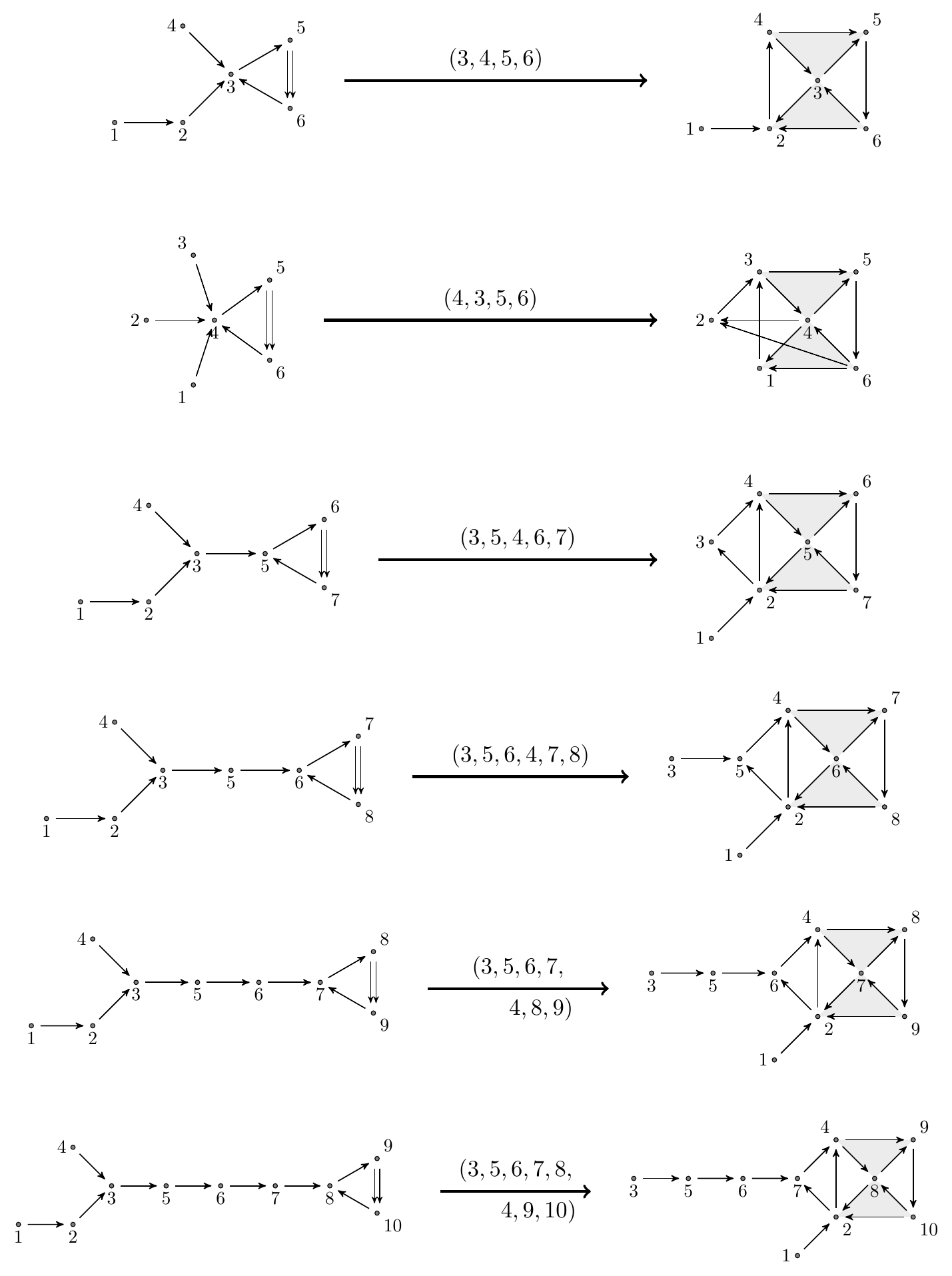}
	\caption{Mutation sequences for each double arrow representative shown in
	Figure~\ref{fig:dareps}. The resulting quiver contains the quiver shown in
	Figure~\ref{fig:non-admissible} which is not mutation-acyclic, so each double
	arrow representative is itself not mutation-acyclic.}
	\label{fig:da-adm-seq}
\end{figure}

\begin{proof}
	Figure~\ref{fig:da-adm-seq} shows a mutation sequence for each of the double
	arrow representatives which results in a quiver containing the quiver shown in
	Figure~\ref{fig:non-admissible} as an induced subquiver.
	Lemma~\ref{lem:adm-nonacyc} shows that this subquiver is not mutation-acyclic,
	so by Lemma~\ref{lem:sub-acyc} each full quiver cannot be mutation-acyclic.
	Hence each double arrow representative is mutation-equivalent to a quiver
	which is not mutation-acyclic, so are themselves not mutation-acyclic.
\end{proof}

\begin{theorem}\label{thm:diff_class}
	With the single exception given in Theorem~\ref{thm:hcs_diff_class} the
	move-classes of minimal mutation-infinite quivers with hyperbolic Coxeter
	simplex diagram representatives or double arrow representatives all belong to
	distinct mutation classes. 
\end{theorem}

\begin{proof}
	Theorem~\ref{thm:hcs_diff_class} shows that all Coxeter simplex move-classes
	of rank greater than $4$ belong to distinct mutation classes, and these
	mutation classes contain acyclic quivers. The double arrow representatives all
	have rank $> 4$ and Lemma~\ref{lem:da-mutacyc} shows that all the double arrow
	move-classes are not mutation-acyclic, and therefore belong to different
	mutation classes to any of the Coxeter simplex classes. Furthermore, all of the double arrow move-classes have a different number of vertices except for the two move-classes of rank 6. However, as shown in Table~\ref{tab3} in the appendix these move-classes have different $B$-matrix rank so they cannot be mutation equivalent. 
\end{proof}
\subsection{Exceptional minimal mutation-infinite quivers.}
We now repeat the argument used for double arrow move-classes to the show that the mutation classes of the minimal mutation-infinite quivers are not mutation-acyclic. This allows us to distinguish the exceptional type classes from the classes whose representative is an orientation of a hyperbolic Coxeter diagram, but it does nothing to distinguish them from the double arrow representatives. We were unable to find an invariant that differentiates the mutation-classes of the double arrow representatives and the exceptional representatives. We were also unable to distinguish some of the exceptional move-classes from other exceptional move-classes of the same rank.
\begin{figure}
		\begin{tikzpicture}[>=stealth',shorten <=5,shorten >=5]
		\def\edgesize{4}
		\tkzDefPoint(0,0){A};
		\tkzDefPoint(\edgesize,0){B};
		\tkzDefPoint(\edgesize,\edgesize){C};
		\tkzDefPoint(0,\edgesize){D};
		\tkzDefPoint(.5*\edgesize,.5*\edgesize){E};
		%\tkzFillPolygon[color=gray!15](A,D,E);
		\tkzFillPolygon[color=gray!15](B,C,E);
		\tkzDrawPoints(A,B,C,D,E)
		\tkzLabelPoint[above left](D){$1$};
		\tkzLabelPoint[above right](C){$2$};
		\tkzLabelPoint[below right](B){$3$};
		\tkzLabelPoint[below left](A){$4$};
		\tkzLabelPoint[right=.2 of E](E){$5$};
		\tkzDrawSegments[->](B,A A,E E,B D,A E,D D,C C,E C,B)
		\tkzLabelSegment[above](D,C){$a$};
		\tkzLabelSegment[above](D,E){$b$};
		\tkzLabelSegment[above](E,C){$c$};
		\tkzLabelSegment[left](D,A){$d$};
		\tkzLabelSegment[right](C,B){$e$};
		\tkzLabelSegment[below](E,A){$f$};
		\tkzLabelSegment[below](E,B){$g$};
		\tkzLabelSegment[below](A,B){$h$};
	\end{tikzpicture}
	\caption{Quiver which does not have an admissible quasi-Cartan companion. The
	vertices are labelled $1, \dotsc, 5$ and the edges $a, \dotsc, h$.
	The non-oriented cycle is shaded gray.}
	\label{fig:non-admissible2}
	\end{figure}
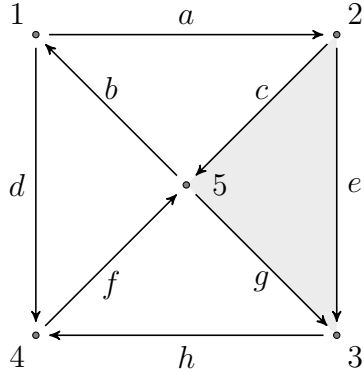
\begin{lemma}\label{lem:adm-nonacyc-2}
	The quiver $Q$ depicted in Figure~\ref{fig:non-admissible2} does not have any
	admissible quasi-Cartan companion, and is hence not mutation-acyclic.
\end{lemma}
\begin{proof}
	Assume the edges and vertices of $Q$ are labelled as in
	Figure~\ref{fig:non-admissible2}, and that the quiver has an admissible
	quasi-Cartan companion. We now use Lemma~\ref{lem:adm-flip} to canonically
	label the edges of $Q$, up to flipping the signs at vertices.
	
	The triangle $(2,3,5)$ in $Q$ is non-oriented, so must have an even number of edges
	labelled~$+$, in particular it must have either one~$+$ or three~$+$'s. By flipping or 		fixing vertices $2,3,$ and $5$ we can ensure that all of the labels $c,e,$ and $g$ are~$-$.

	Now the oriented triangles $(1, 2, 5)$ and $(3, 4, 5)$ require an odd
	number of~$+$'s, while they each already have one~$-$. Hence one of $a$
	or $b$ is~$+$ and the other is~$-$, and the same for $h$ and $f$. By flipping
	$1$ and $4$ we can choose that $a$ and $h$ are~$+$ while $b$ and $f$ are~$-$.
	This leaves the oriented triangle $(1, 4, 5)$ with $d$ labelled by a~$+$ since it 	requires
	an odd number of positives. We have now fixed or
	flipped all vertices and have assigned the only possible labels to all the
	edges, up to flipping the signs at vertices.

	However we now have a non-oriented cycle $(1,2,3,4)$ which contains three~$+$'s and
	one~$-$, which gives a contradiction as a non-oriented cycle must have an odd
	number of positive edges in an admissible companion. Hence $Q$ cannot admit an
	admissible quasi-Cartan companion and so by Theorem~\ref{thm:adm-mutacyc}
	cannot be mutation-acyclic.
\end{proof}
\begin{figure}
\includegraphics{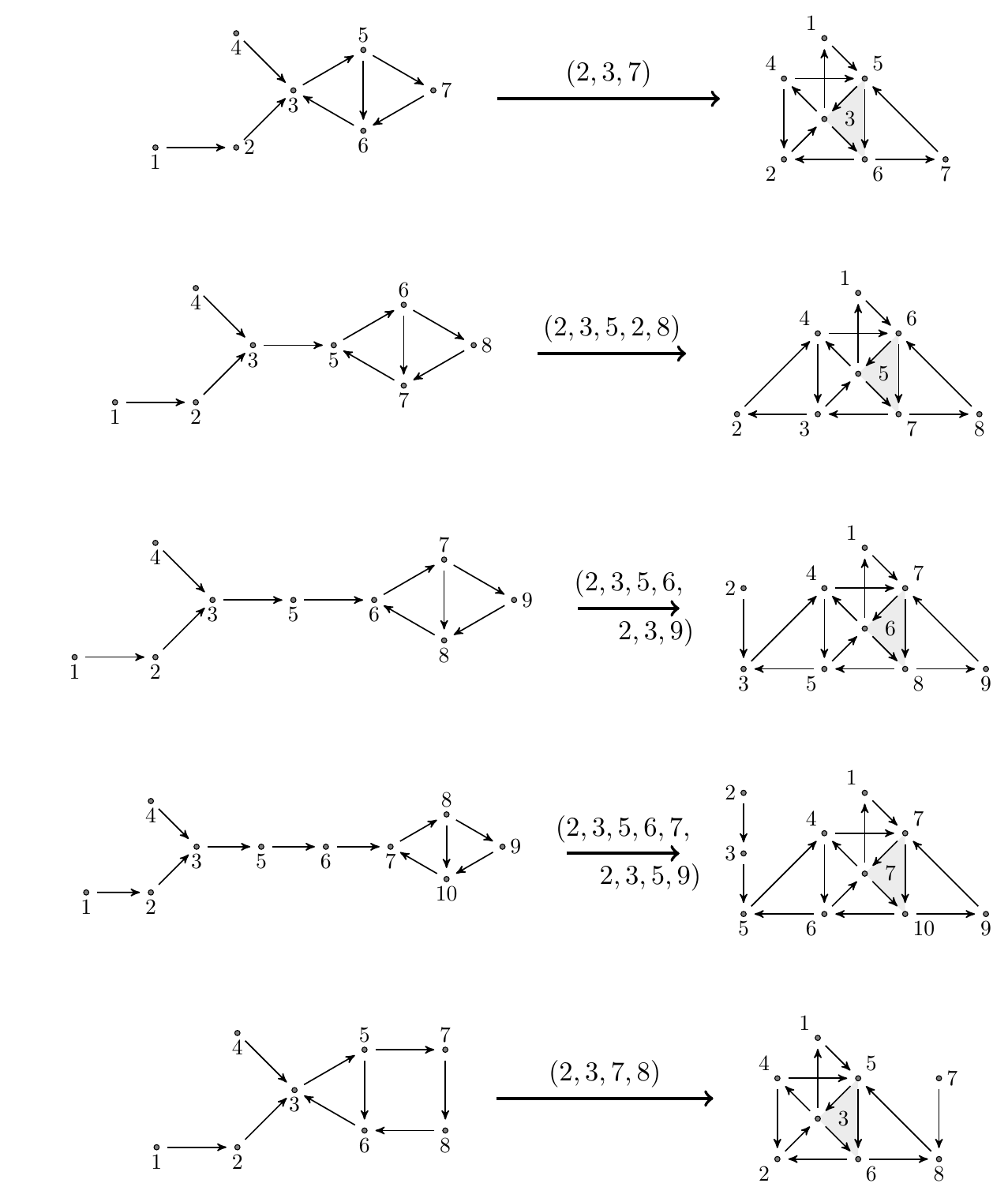}
\caption{Mutation sequences for the first 5 exceptional representative shown in
	Figure~\ref{fig:excreps}. The other 5 are shown in Figure~\ref{fig:exc-adm-seq-2}. The resulting quiver contains the quiver shown in
	Figure~\ref{fig:non-admissible2} which is not mutation-acyclic, so each exceptional representative is itself not mutation-acyclic.}
	\label{fig:exc-adm-seq-1}
	\end{figure}
	\begin{figure}
\includegraphics{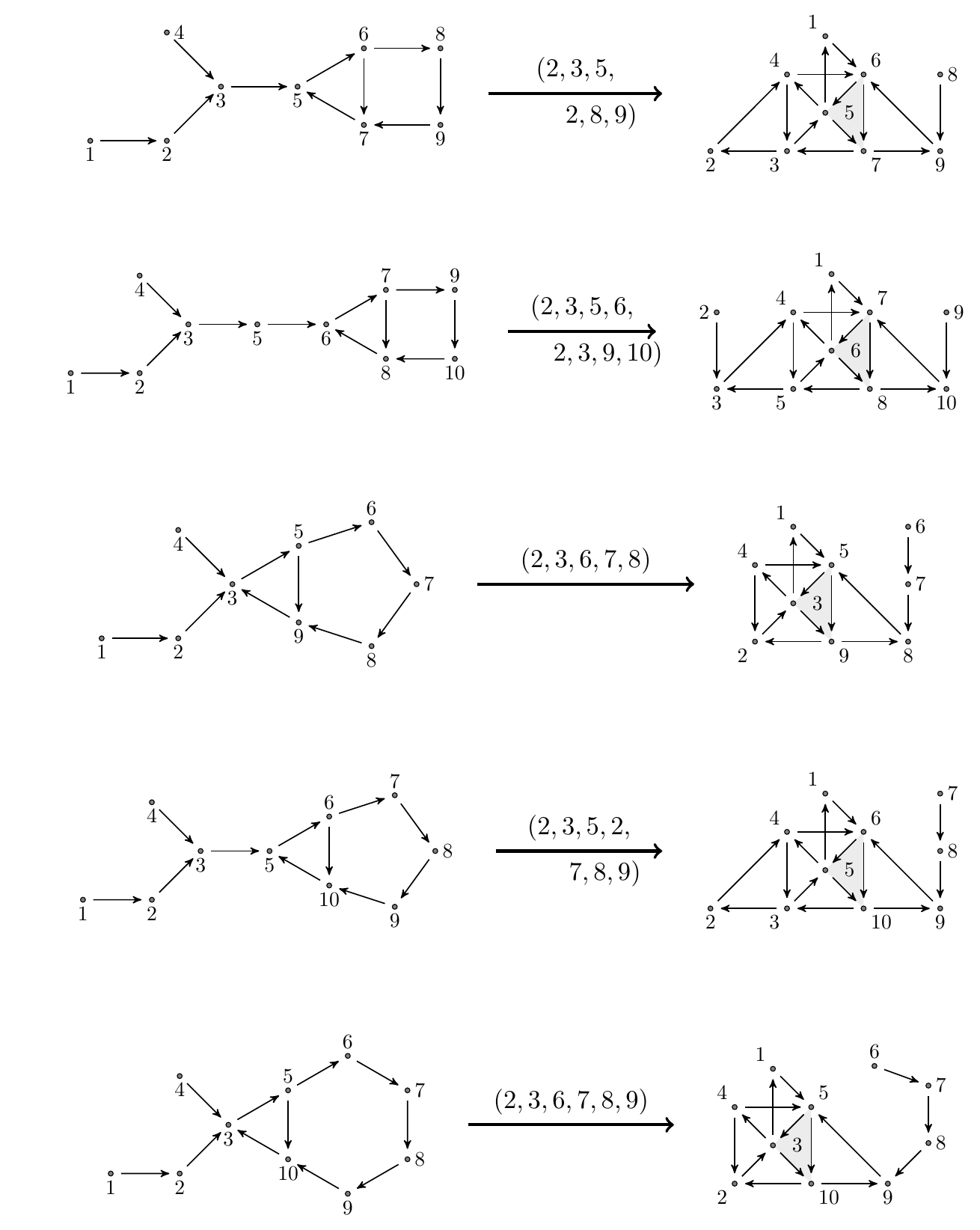}
\caption{The remaining 5 exceptional representatives not shown in Figure~\ref{fig:exc-adm-seq-1}.}
	\label{fig:exc-adm-seq-2}
\end{figure}
\begin{lemma}\label{lem:ex-mutacyc}
	Each exceptional representative is mutation-equivalent to a quiver containing
	the quiver depicted in Figure~\ref{fig:non-admissible2} as an induced
	subquiver. Hence each exceptional move-class is not mutation-acyclic.
\end{lemma}
\begin{proof}
The proof is identical to that of Lemma~\ref{lem:da-mutacyc} with Lemma~\ref{lem:adm-nonacyc-2} in place of Lemma~\ref{lem:adm-nonacyc} and Figures~\ref{fig:exc-adm-seq-1} and~\ref{fig:exc-adm-seq-2} in place of Figure~\ref{fig:da-adm-seq}.
\end{proof}

\section{Maximal green sequences for minimal mutation-infinite quivers}\label{sec:mmi-mgs}
Building on~\cite{mills} we consider when minimal mutation-infinite quivers have maximal green sequences.

\subsection{Rank 3 quivers}
If $Q$ is an acyclic rank 3 quiver then it has a maximal green sequence by
Theorem~\ref{thm:acyclic} so we only consider those quivers which are oriented 3-cycles. 
Let $Q_{a,b,c}$ denote such a quiver with vertices $1,2,$ and $3$ and $a$ edges $1 \rightarrow 2$, $b$ edges $2 \rightarrow 3$ and $c$ edges $3 \rightarrow 1$. 
\begin{theorem}\cite{muller,S1}\label{thm:rank3_no_mgs}%[Theorem 2.3.1,Theorem 1.2]
If $a,b$ and $c \geq 2$, then the quiver $Q_{a,b,c}$ does not admit a maximal green sequence. 
\end{theorem}

\begin{theorem}
If any of $a,b$ or $c$ are equal to 1, then $Q_{a,b,c}$ has a maximal green sequence. 
\end{theorem}
\begin{proof}
%Note that if $a,b$ or $c$ is zero then the quiver is acyclic and has a maximal green sequence. 
Without loss of generality assume that $a=1$. If $b > c$ then $(2,1,3,2)$ is a maximal green sequence for $Q_{1,b,c}$. If $c > b$ then $(2,3,1,2)$ is a maximal green sequence for $Q_{1,b,c}$. If $c=b$ then either mutation sequence is a maximal green sequence for $Q_{1,b,c}$. 
\end{proof}

All mutation-infinite quivers of rank 3 are minimal mutation-infinite, so by the
above not all such quivers admit a maximal green sequence. In Section~\ref{sec:mut_class_mgs} we show that all of the quivers in the mutation class of a mutation-infinite quiver of rank 3 that admit a maximal green sequence form a finite connected subgraph of the quiver exchange graph. 

\subsection{Higher rank quivers}
There are an infinite number of rank 3 minimal mutation-infinite quivers, some
of which have maximal green sequences and some do not. In contrast, only a
finite number of higher rank minimal mutation-infinite quivers exist and we will
show that all such quivers do in fact admit a maximal green sequence.

\begin{lemma}\label{lem:no_bad_subquivers}
Let $Q$ be a minimal mutation-infinite quiver. Then $Q$ does not contain a
subquiver that arises from a triangulation of a once-punctured closed surface, or one that is in the mutation class of $\mathbb{X}_7$. 
\end{lemma}
\begin{proof}
	The mutation class of the $\mathbb{X}_7$ quiver consists of two quivers, one
	contains three double arrows while the other contains six vertices which are each the
	source of two arrows and the target of two arrows. See Figure~\ref{fig:x7}.
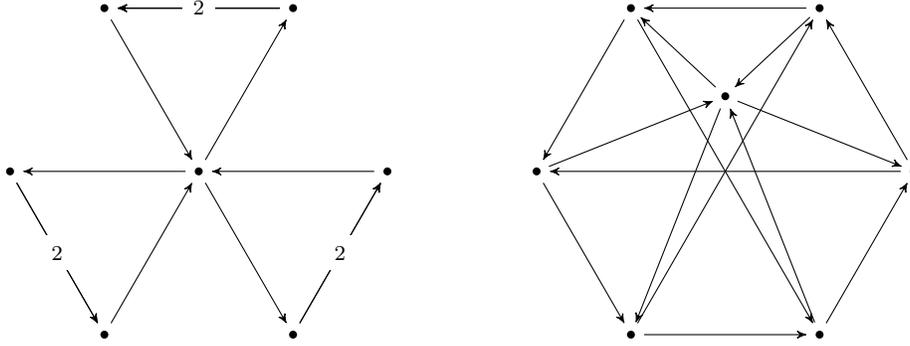
\begin{figure}
\begin{tikzpicture}[font = \tiny]
	\node[draw=none,minimum size=5cm,regular polygon,regular polygon sides=6] at (0,0) (a) {};
	\foreach \x in {1,2,...,6}
		\fill (a.corner \x) circle[radius=1.5pt]; 
	\draw[->, shorten <=5, shorten >=5, >=stealth'] (a.corner 1) to node[fill=white] {2} (a.corner 2);
	\fill (0,0) circle (1.5pt); 
	 
	\draw[->, shorten <=5, shorten >=5, >=stealth'] (a.corner 1) to node[fill=white] {2} (a.corner 2);
	\draw[->, shorten <=5, shorten >=5, >=stealth'] (0,0) to (a.corner 1);
	\draw[->, shorten <=5, shorten >=5, >=stealth'] (a.corner 2) to (0,0);
	\draw[->, shorten <=5, shorten >=5, >=stealth'] (a.corner 1) to node[fill=white] {2} (a.corner 2);

	\draw[->, shorten <=5, shorten >=5, >=stealth'] (a.corner 5) to node[fill=white] {2} (a.corner 6);
	\draw[->, shorten <=5, shorten >=5, >=stealth'] (0,0) to (a.corner 5);
	\draw[->, shorten <=5, shorten >=5, >=stealth'] (a.corner 6) to (0,0);
	\draw[->, shorten <=5, shorten >=5, >=stealth'] (a.corner 5) to node[fill=white] {2} (a.corner 6);

	\draw[->, shorten <=5, shorten >=5, >=stealth'] (a.corner 3) to node[fill=white] {2} (a.corner 4);
	\draw[->, shorten <=5, shorten >=5, >=stealth'] (0,0) to (a.corner 3);
	\draw[->, shorten <=5, shorten >=5, >=stealth'] (a.corner 4) to (0,0);
	\draw[->, shorten <=5, shorten >=5, >=stealth'] (a.corner 3) to node[fill=white] {2} (a.corner 4);

	\node[draw=none,minimum size=5cm,regular polygon,regular polygon sides=6] at (7,0) (a) {};
	\foreach \x in {1,2,...,6}
		\fill (a.corner \x) circle[radius=1.5pt]; 
	\fill (7,1) circle (1.5pt);
	\draw[->, shorten <=5, shorten >=5, >=stealth'] (a.corner 1) to (a.corner 2);
	\draw[->, shorten <=5, shorten >=5, >=stealth'] (a.corner 2) to (a.corner 3);
	\draw[->, shorten <=5, shorten >=5, >=stealth'] (a.corner 3) to (a.corner 4);
	\draw[->, shorten <=5, shorten >=5, >=stealth'] (a.corner 4) to (a.corner 5);
	\draw[->, shorten <=5, shorten >=5, >=stealth'] (a.corner 5) to (a.corner 6);
	\draw[->, shorten <=5, shorten >=5, >=stealth'] (a.corner 6) to (a.corner 1);
	\draw[->, shorten <=5, shorten >=5, >=stealth'] (a.corner 1) to (7,1);
	\draw[->, shorten <=5, shorten >=5, >=stealth'] (a.corner 3) to (7,1);
	\draw[->, shorten <=5, shorten >=5, >=stealth'] (a.corner 5) to (7,1);
	\draw[->, shorten <=5, shorten >=5, >=stealth'] (7,1) to (a.corner 2);
	\draw[->, shorten <=5, shorten >=5, >=stealth'] (7,1) to (a.corner 4);
	\draw[->, shorten <=5, shorten >=5, >=stealth'] (7,1) to (a.corner 6);
	\draw[->, shorten <=5, shorten >=5, >=stealth'] (a.corner 4) to (a.corner 1);
	\draw[->, shorten <=5, shorten >=5, >=stealth'] (a.corner 2) to (a.corner 5);
	\draw[->, shorten <=5, shorten >=5, >=stealth'] (a.corner 6) to (a.corner 3);
\end{tikzpicture}
\caption{The quiver $\mathbb{X}_7$ is on the left and the other quiver in its mutation class is on the right.}\label{fig:x7}
\end{figure}

	Ladkani shows in~\cite[Prop.\ 3.6.]{ladkani-numarrows} that for any quiver
	arising from a triangulation of a once-punctured surface without boundary each
	vertex is the source of two arrows and the target of two arrows. In the genus
	1 surface case we get the Markov quiver with 3 vertices. In the case of the
	genus 2 surface each quiver has 9 such vertices.

	The only minimal mutation-infinite quivers containing any double edges are the
	double arrow representatives, which each contain a single double edge. It can
	also be seen through an exhaustive search that no minimal mutation-infinite
	quiver contains more than 5 vertices which are each adjacent to 4 or more
	arrows. Hence no minimal mutation-infinite quivers contain subquivers from
	once-punctured surfaces or from $\mathbb{X}_7$.
\end{proof}

\begin{corollary}\label{cor:mgs_subquivers}
Let $Q$ be a minimal mutation-infinite quiver. Every induced subquiver of $Q$ has a maximal green sequence. 
\end{corollary}
\begin{proof}
Follows from Theorem~\ref{thm:mu_finite_mgs} and Lemma~\ref{lem:no_bad_subquivers}.
\end{proof}

\begin{theorem}\label{thm:mmi_mgs}
Suppose $Q$ is a minimal mutation-infinite quiver of rank at least 4. Then $Q$
has a maximal green sequence.
\end{theorem}
\begin{proof}
Let $Q$ be a minimal mutation-infinite quiver which is the $t$-colored direct
sum of two induced subquivers of $Q$. Then by Theorem~\ref{thm:directsum-mgs}
and Corollary~\ref{cor:mgs_subquivers} the quiver $Q$ has a maximal green sequence.

In particular, if a minimal mutation-infinite quiver contains either a sink or a
source then it has a maximal green sequence. Similarly if it can be decomposed
into two disjoint induced subquivers joined by a single arrow then it has a maximal green
sequence.

There are only $42$ minimal mutation-infinite quivers which cannot be written as
a $t$-colored direct sum\footnote{Images of which can be found at
\url{https://www.maths.dur.ac.uk/users/j.w.lawson/mmi/quivers/non-direct-sum/}}.
Of these, $35$ end in a 3-cycle and so admit a maximal green sequence by
Theorem~\ref{thm:triangleend} and Corollary~\ref{cor:mgs_subquivers}.

This leaves $7$ minimal mutation infinite quivers which cannot be written as a
direct sum or end in a 3-cycle. Six of these 7 quivers have a similar structure and we will call these quivers $\Theta_n$ for $n=4,5,6,7,8,9$. A picture of this family of quivers is given on the left of Figure~\ref{fig:mgs_quivers}. For a suitable $n$, the quiver $\Theta_n$ has a maximal green sequence $(2, 3, \dots, n, 1, 2)$.
The final quiver, appearing on the right of Figure~\ref{fig:mgs_quivers}, has a maximal green sequence $(3, 1, 2, 5, 6, 4, 3)$.
\end{proof}

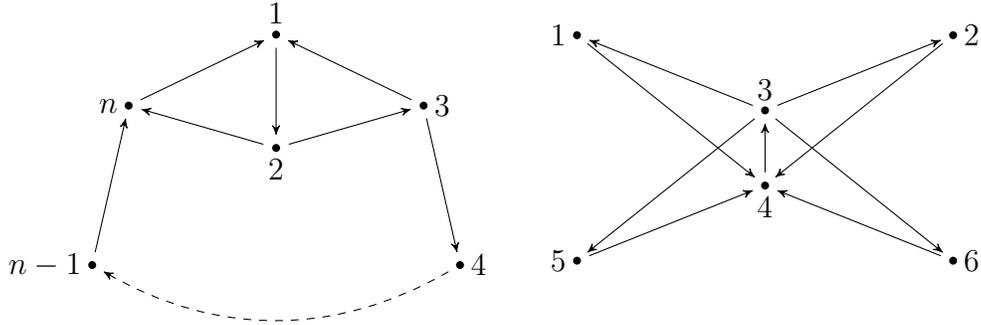
\begin{figure}
\begin{tikzpicture}
\node[draw=none,minimum size=5cm,regular polygon,regular polygon sides=7] at (0,0) (a) {};
\foreach \x in {1,2,3,6,7} \fill (a.corner \x) circle[radius=1.5pt];
\fill (0,1) circle (1.5pt);

\draw[->, shorten <=5, shorten >=5, >=stealth'] (a.corner 1) node[above]{1} to (0,1) node[below]{2};
\draw[->, shorten <=5, shorten >=5, >=stealth'] (0,1) to (a.corner 7) node[right]{3};
\draw[->, shorten <=5, shorten >=5, >=stealth'] (a.corner 7) to (a.corner 1);
\draw[->, shorten <=5, shorten >=5, >=stealth'] (a.corner 7) to (a.corner 6) node[right]{4};
\draw[->,dashed,shorten <=5, shorten >=5,>=stealth'] (a.corner 6) to[bend left] (a.corner 3);
\draw[->, shorten <=5, shorten >=5, >=stealth'] (a.corner 3) node[left]{$n-1$} to (a.corner 2) node[left]{$n$};
\draw[->, shorten <=5, shorten >=5, >=stealth'] (a.corner 2) to (a.corner 1);
\draw[->, shorten <=5, shorten >=5, >=stealth'] (0,1) to (a.corner 2);

\coordinate (a1) at (4,2.5); 
\coordinate (a2) at (9,2.5); 
\coordinate (a3) at (6.5,1.5); 
\coordinate (a4) at (6.5,.5); 
\coordinate (a5) at (4,-.5); 
\coordinate (a6) at (9,-.5);
\foreach \x in {1,2,...,6} \fill (a\x) circle[radius=1.5pt];

\draw[->, shorten <=5, shorten >=5, >=stealth'] (a1) node[left]{1} to (a4);
\draw[->, shorten <=5, shorten >=5, >=stealth'] (a2) node[right]{2} to (a4);
\draw[->, shorten <=5, shorten >=5, >=stealth'] (a5) node[left]{5} to (a4);
\draw[->, shorten <=5, shorten >=5, >=stealth'] (a6) node[right]{6} to (a4);
\draw[->, shorten <=5, shorten >=5, >=stealth'] (a4) node[below]{4} to (a3);
\draw[->, shorten <=5, shorten >=5, >=stealth'] (a3) node[above]{3} to (a1);
\draw[->, shorten <=5, shorten >=5, >=stealth'] (a3) to (a2);
\draw[->, shorten <=5, shorten >=5, >=stealth'] (a3) to (a5);
\draw[->, shorten <=5, shorten >=5, >=stealth'] (a3) to (a6);
\end{tikzpicture}
\caption{Minimal mutation-infinite quivers that cannot be written as a $t$-colored direct sum and do not end in a 3-cycle.}\label{fig:mgs_quivers}
\end{figure}

\section{A=U for minimal mutation-infinite quivers}\label{sec:mmi-upper}
For completeness we recall the result about the rank 3 quivers.
\begin{theorem}\cite[Theorem 1.3]{llm}
Let $\mathcal{A}$ be a rank 3 cluster algebra. The cluster algebra $\mathcal{A}$ is equal to its upper cluster algebra if and only if it is acyclic. 
\end{theorem}
In the case of higher rank minimal mutation-infinite quivers we show that they are all Louise and it follows that the cluster algebra that they generate is equal to its upper cluster algebra. 
\begin{theorem}\label{thm:mmi_Louise}
If $Q$ is a minimal mutation-infinite quiver of rank at least 4, then $Q$ is Louise.
%with move-class representative an
%orientation of a hyperbolic Coxeter simplex diagram, $\mathcal{A}(Q)$ equals its upper cluster algebra. 
\end{theorem}
\begin{proof}

All minimal mutation-infinite quiver representatives arising as an orientation
of a hyperbolic Coxeter simplex diagram are acyclic apart from two orientations
of the fully connected rank 4 quiver, shown as the fifth and sixth quivers in
Figure~\ref{fig:hcsreps}. The sixth quiver is mutation-equivalent to an
acyclic quiver by mutating at the top vertex.

All quivers in a move-class are mutation-equivalent and hence each minimal
mutation-infinite quiver $Q$ in all but one Coxeter diagram move-class is
mutation-equivalent to an acyclic quiver. Therefore the cluster algebra is
acyclic and hence Louise. 

The remaining case for the hyperbolic Coxeter representatives is the move-class of the fifth rank 4 quiver in Figure~\ref{fig:hcsreps}. We will show that it is locally acyclic so the claim follows from Theorem~\ref{thm:LA_AU}. Consider the representative quiver $R$ given in Figure~\ref{fig:mut-acyclic-rep}. Two vertices of the quiver labeled $i$ and $j$. The edge $1 \rightarrow 2$ is a separating edge. The quivers $R[R_0 \setminus \{2\}]$ and $R[R_0 \setminus \{1,2\}]$ are acyclic and hence Louise. Mutating the quiver $R[R_0 \setminus \{1\}]$ at $j$ produces an acyclic quiver, which again shows that it is Louise. Therefore the quiver $R$ is Louise.

The proof for minimal mutation-infinite quivers with move-class either of double arrow type or exceptional type is identical to the argument given above with the following choice of vertices for $1$ and $2$. In the case of the double arrow representatives we take $2$ to be the vertex opposite the double arrow and $1$ to be an adjacent vertex that is not incident to the double arrow. Similarly, for the exceptional representatives we take $2$ to be the leftmost vertex of the 3-cycle and $1$ to be an adjacent vertex that is not a part of the 3-cycle. 
\end{proof}
%\begin{theorem}\label{thm:da_er_louise}
%If $Q$ is a minimal mutation-infinite quiver with move-class either of double arrow type or exceptional type, then $Q$ is Louise.
%\end{theorem}
%\begin{proof}
%
%\end{proof}
\begin{corollary}\label{cor:mmi_upper}
If $Q$ is a minimal mutation-infinite quiver of rank at least 4, then the cluster algebra $\mathcal{A}(Q)$ is equal to its upper cluster algebra. 
\end{corollary}
\begin{proof}
Follows from Theorem~\ref{thm:LA_AU} and Theorem~\ref{thm:mmi_Louise}.
\end{proof}

\section{Quivers in the mutation class with maximal green sequences}\label{sec:mut_class_mgs}
Muller showed in \cite{muller} that in general the existence of a maximal green sequence is not mutation-invariant. This motivates the question of which quivers in a mutation class have a maximal green sequence. 

 Let $\Psi(Q)$ denote the (possibly empty) subgraph of the unlabelled quiver exchange graph $\Gamma(Q)$ consisting of the quivers that have a maximal green sequence. Rephrasing Theorem~\ref{thm:mu_finite_mgs} we have the following result. 
\begin{theorem}\label{thm:rank3_eg}
Let $Q$ be a mutation-finite quiver. Either the graph $\Psi=\Gamma$ or $\Psi$ is empty. 
\end{theorem}

Although there are infinitely many rank 3 quivers they only produce finitely many different exchange graphs.
\begin{theorem}\label{thm:finite_rank3_mgs}
Let $Q$ be a rank 3 quiver. If $Q$ is mutation-acyclic, then $\Psi$ is one the 7 graphs in Figures~\ref{fig:r3-lin-mgs-graphs} and~\ref{fig:r3-tri-mgs-graphs}, otherwise, $\Psi(Q)$ is empty. It then follows that the number of quivers in the mutation class of $Q$ is bounded and in particular $|\Psi| \leq 6$.
\end{theorem}
\begin{proof}
If $Q$ is not mutation-acyclic then by \cite[Theorem 1.2]{bbh} its entire mutation class consists of quivers of the form $Q_{a,b,c}$ with $a,b,c \geq 2$. Therefore by Theorem~\ref{thm:rank3_no_mgs} no quiver in the mutation class has a maximal green sequence and we may conclude that $\Psi(Q)$ is empty.

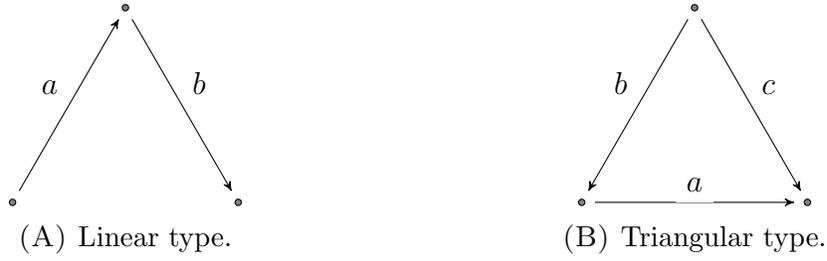
\begin{figure}[t]
\begin{minipage}{.45\linewidth}
	\centering
	\begin{tikzpicture}
		\tkzDefPoint(0,0){A};
		\tkzDefPoint(1.5,2.59){B};
		\tkzDefPoint(3,0){C};
		\tkzDrawPoints(A,B,C)
		\draw[->, shorten >=5, shorten <=5, >=stealth'] (A) to node[above left,fill=white]{$a$} (B);
		\draw[->, shorten >=5, shorten <=5, >=stealth'] (B) to node[above right,fill=white]{$b$} (C); 
	\end{tikzpicture}
	\subcaption{Linear type.}
	\label{fig:r3-acyclic-lin}
\end{minipage}
\begin{minipage}{.45\linewidth}
	\centering
	\begin{tikzpicture}
		\tkzDefPoint(0,0){A};
		\tkzDefPoint(3,0){B};
		\tkzDefPoint(1.5,2.59){C};
		\tkzDrawPoints(A,B,C)
		\draw[->, shorten >=5, shorten <=5, >=stealth'] (C) to node[above left,fill=white]{$b$} (A);
		\draw[->, shorten >=5, shorten <=5, >=stealth'] (C) to node[above right, fill=white]{$c$} (B); 
		\draw[->, shorten >=5, shorten <=5, >=stealth'] (A) to node[above,fill=white]{$a$} (B); 
	\end{tikzpicture}
	\subcaption{Triangular type.}
	\label{fig:r3-acyclic-tri}
\end{minipage}
\caption{The two types of (connected) acyclic rank 3 quivers which give
	different exchange graphs, for $a,b,c \in \mathbb{Z}_{>0}$.}
\label{fig:r3-acyclic}
\end{figure}

If $Q$ is mutation-acyclic then its mutation class must contain an acyclic
quiver $R$ in Figure~\ref{fig:r3-acyclic} for some $a,b,c \in \mathbb{Z}_{>
0}$. Clearly $\Psi(Q)=\Psi(R)$ so we may proceed by showing that $\Psi(R)$ is
one of the exchange graphs in
Figures~\ref{fig:r3-lin-mgs-graphs} and~\ref{fig:r3-tri-mgs-graphs}. Every vertex of
$\Gamma(R)$ that is not displayed in one of these figures corresponds to a
quiver $Q_{i,j,k}$ for some $i,j,k \geq 2$. Indeed, it is easy to check that
for any one step mutation that does not appear in these graphs the resulting
quiver is of the claimed form.
To see the claim for a longer mutation sequence, suppose that we mutate a
quiver in $\Psi(R)$ at vertex $v$ to obtain the quiver $Q_{i,j,k} \not \in
\Psi(R)$. Then according to Lemma 2.4 in~\cite{assem} if we mutate at $\ell \in
(Q_{i,j,k})_0 \setminus \{v\}$ we have
\[\mu_\ell (Q_{i,j,k}) = Q_{i',j',k'}\]
with $i' \geq i, j' \geq j, $ and $k' \geq k$. Therefore any mutation sequence
of length two yields a quiver without a maximal green sequence. 
Iterating this argument we see for any mutation sequence that moves along $\Gamma(R)$ outside of the region $\Psi(R)$ we will always obtain a quiver that is a 3-cycle which does not have a maximal green sequence since the number of edges will all be greater than or equal to 2.

\begin{figure}
	\begin{minipage}{\textwidth}
		\centering
		\includegraphics[scale=.65]{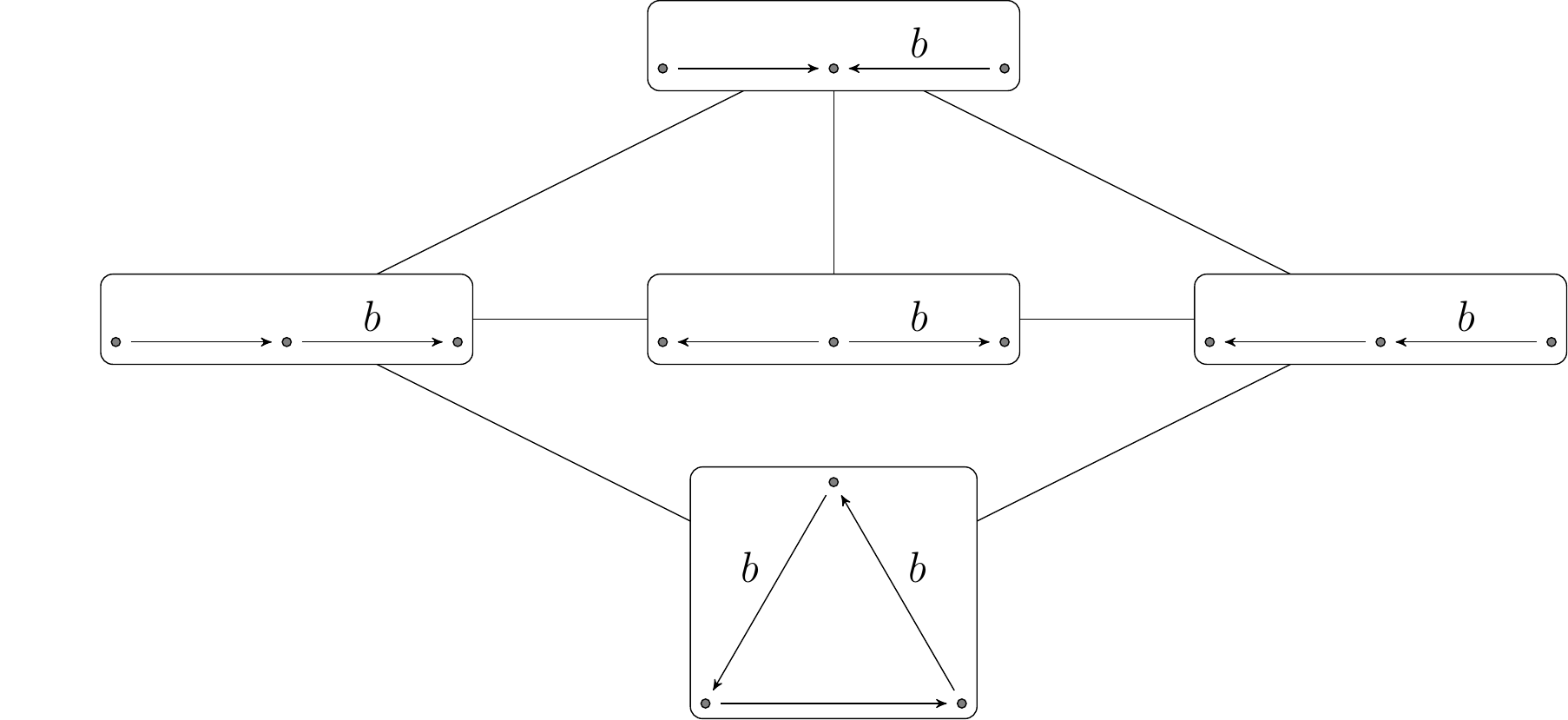}
		\subcaption{$b > a = 1$}
		\label{fig:rank3_case_1}
	\end{minipage}
	\begin{minipage}{\textwidth}
		\centering
		\includegraphics[scale=.65]{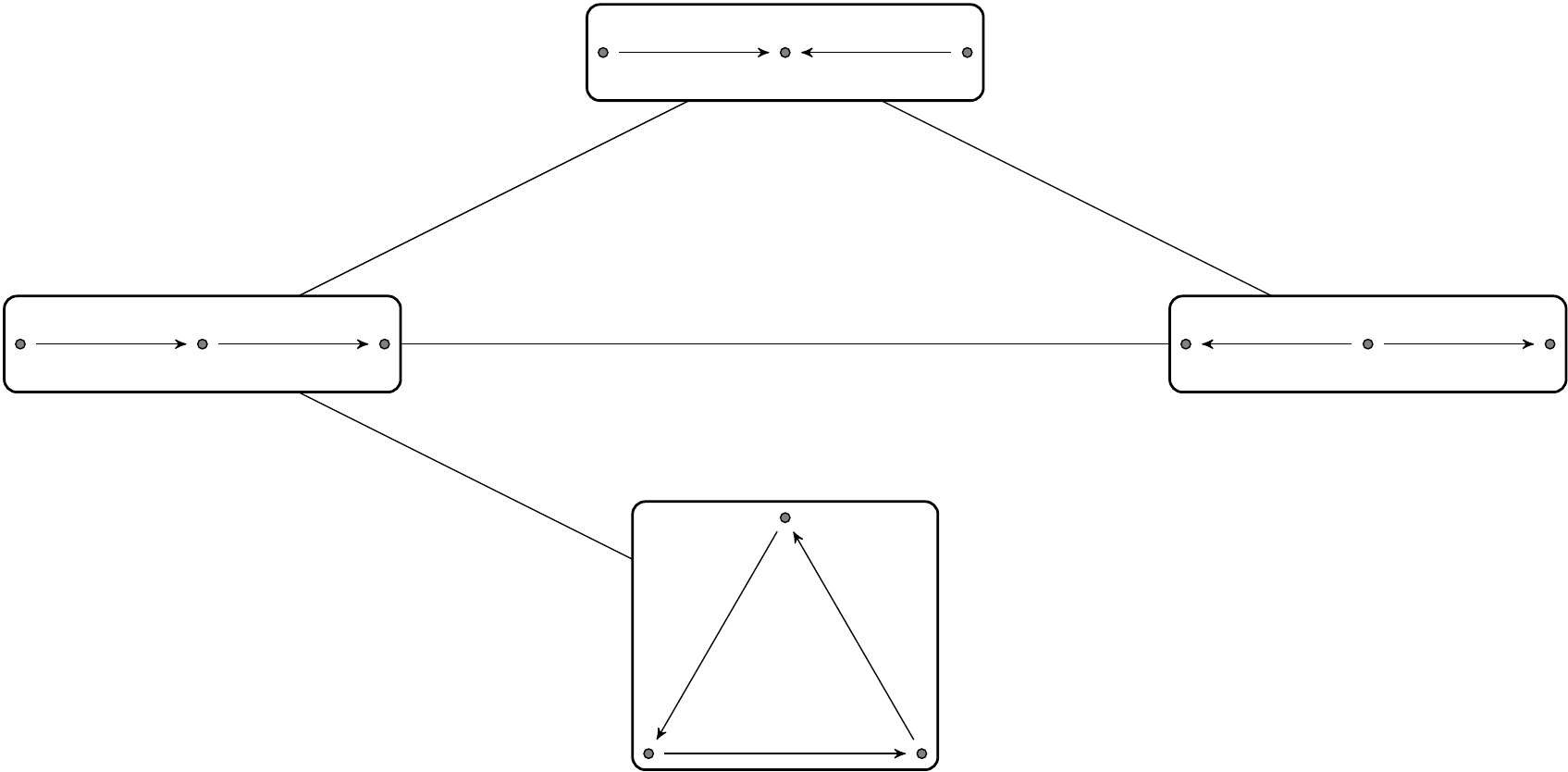}
		\subcaption{$a = b = 1$}
		\label{fig:rank3_case_2}
	\end{minipage}
	\\[6pt]
	\begin{minipage}{\textwidth}
		\centering
		\includegraphics[scale=.55]{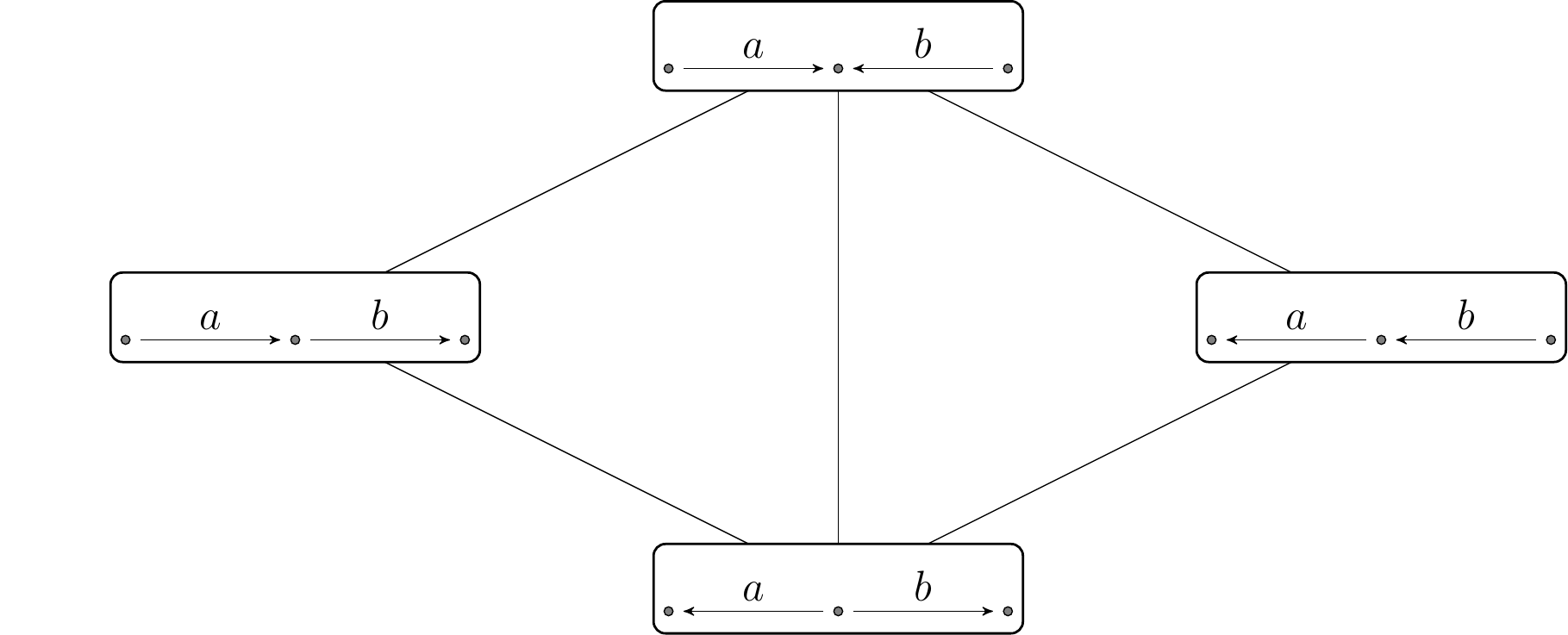}
		\subcaption{$b > a > 1$}
		\label{fig:rank3_case_3}
	\end{minipage}
	\begin{minipage}{\textwidth}
		\centering
		\includegraphics[scale=.55]{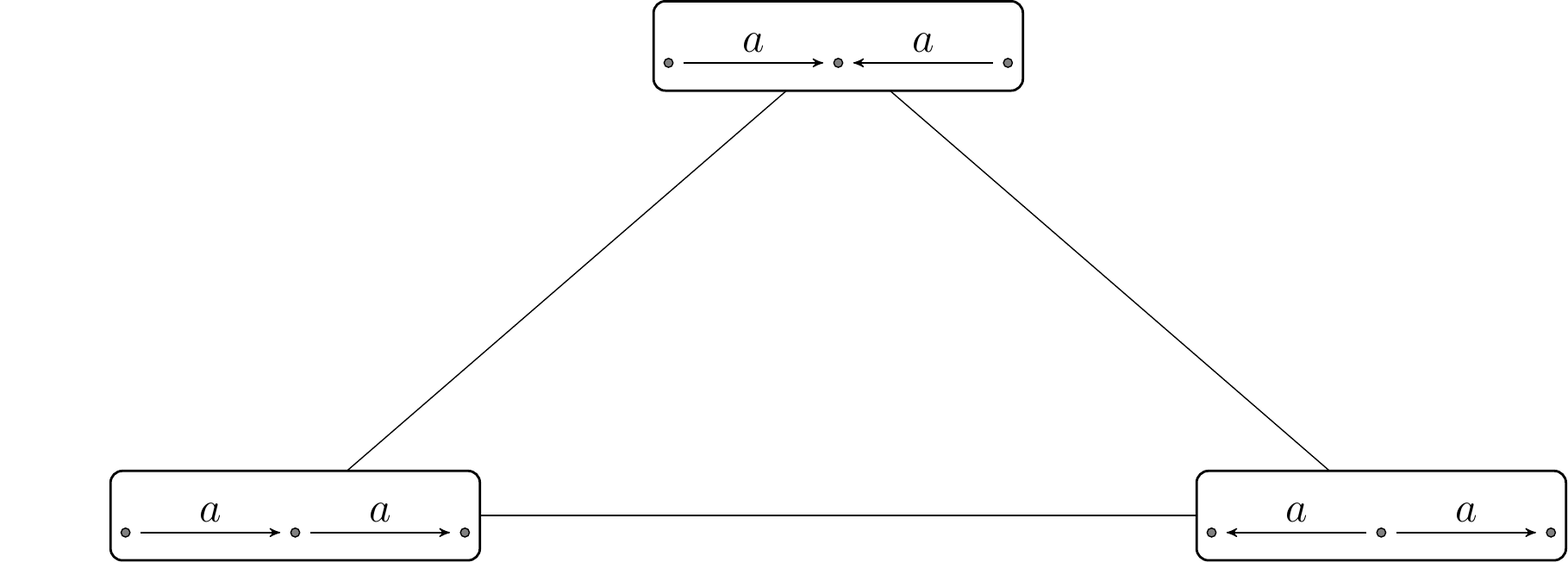}
		\subcaption{$b = a > 1$}
		\label{fig:rank3_case_4}
	\end{minipage}
	\caption{Subgraphs of the quiver exchange graphs showing only those quivers
	with maximal green sequences for the quiver appearing in
	Figure~\ref{fig:r3-acyclic-lin}.}
	\label{fig:r3-lin-mgs-graphs}
\end{figure}

Suppose $R$ is of the linear type in Figure~\ref{fig:r3-acyclic-lin}, then there are four possibilities for $\Psi(R)$, which are determined by the values of $a$ and $b$. The four cases are: \begin{enumerate}
\item$b > a = 1$; \item$b = a = 1$; \item $b >a > 1$; \item $b = a > 1$,
\end{enumerate}
with corresponding graphs $\Psi(R)$ shown in Figure~\ref{fig:r3-lin-mgs-graphs}. 

\begin{figure}
	\begin{minipage}{.48\textwidth}
		\centering
		\includegraphics[scale=.6]{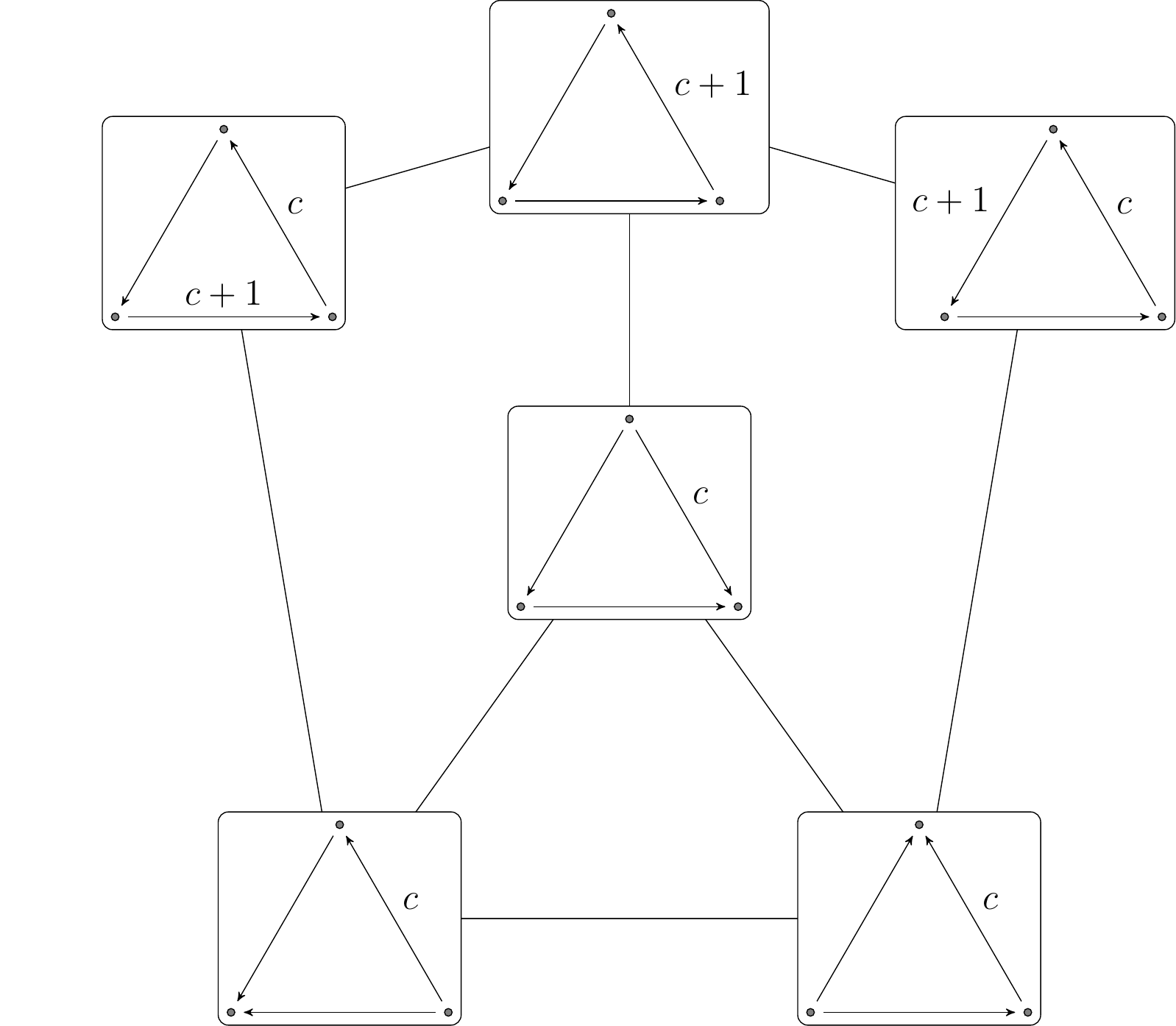}
		\subcaption{$a = b = 1;\; c > 1$}
		\label{fig:rank3_case_5}
	\end{minipage}
	\begin{minipage}{.48\textwidth}
		\centering
		\includegraphics[scale=.6]{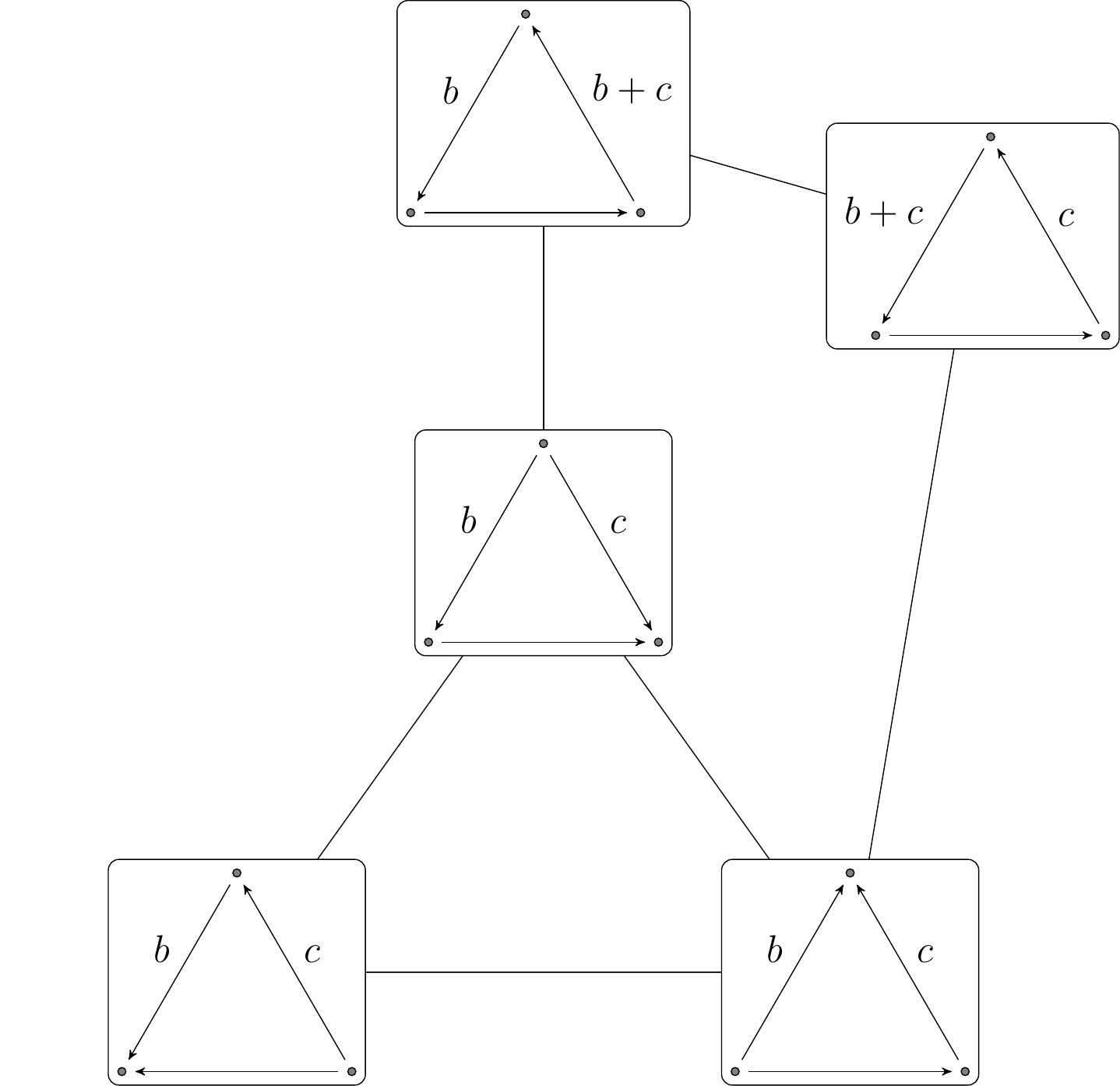}
		\subcaption{$a = 1;\; b, c > 1$}
		\label{fig:rank3_case_7}
	\end{minipage}
	\\[10pt]
	\begin{minipage}{.45\textwidth}
		\centering
		\includegraphics[scale=.6]{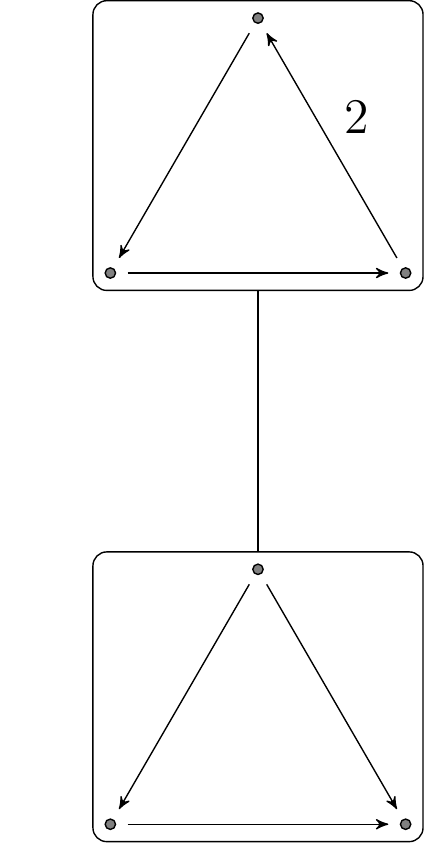}
		\subcaption{$a = b = c = 1$}
		\label{fig:rank3_case_6}
	\end{minipage}
	\begin{minipage}{.45\textwidth}
		\centering
		\includegraphics[scale=.6]{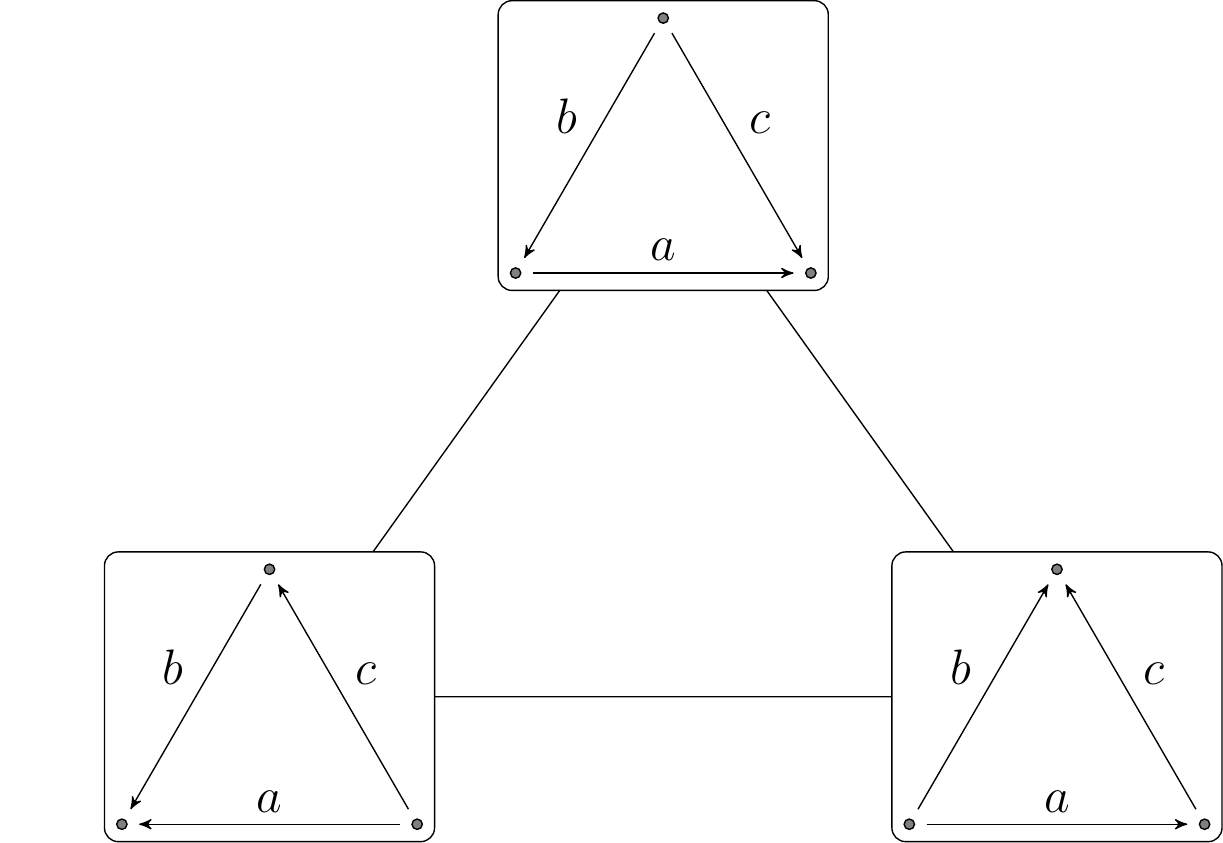}
		\subcaption{$a, b, c > 1$}
		\label{fig:rank3_case_8}
	\end{minipage}
	\caption{Subgraphs of the quiver exchange graphs showing only those quivers
	with maximal green sequences for the quiver appearing in
	Figure~\ref{fig:r3-acyclic-tri}.}
	\label{fig:r3-tri-mgs-graphs}
\end{figure}

If $R$ is of the triangular type in Figure~\ref{fig:r3-acyclic-tri} we again have four cases to consider:
\begin{enumerate}
\item $ b = a = 1$, $c >1$;
\item $b,c > 1, a = 1$;
\item $c = b = a = 1$;
\item $a,b,c >1$,
\end{enumerate}
with corresponding graphs shown in Figure~\ref{fig:r3-tri-mgs-graphs}. Note that Figure~\ref{fig:rank3_case_4} and Figure~\ref{fig:rank3_case_8} give isomorphic graphs for $\Psi(R)$. Hence, there are 7 possibilities for a mutation-acyclic rank 3 quiver. The largest such graph has 6 vertices. 
\end{proof}

We now prove a result similar to Theorem~\ref{thm:rank3_eg} for rank 4 minimal mutation-infinite quivers. When constructing $\Phi$ from $\Gamma$ we can use Theorem~\ref{thm:subquiver} and Theorem~\ref{thm:rank3_no_mgs} to remove several vertices from $\Gamma$. However, this approach is not sufficient to eliminate all quivers that do not admit a maximal green sequence. To handle this we introduce the notion of a good mutation sequence, which is a slight generalization of a maximal green sequence, to show that other quivers appearing in $\mut(Q)$ do not have a maximal green sequence. 
\begin{definition}
A vertex $k$ of a quiver $Q$ is called a \hdef{good vertex} if\begin{enumerate}
\item $k$ is not the head of a multiple edge;
\item  and $\mu_k(Q)$ does not contain an induced subquiver that does not admit a maximal green sequence. 
\end{enumerate}
A mutation sequence is called a \hdef{good sequence} if at every step of the mutation sequence we mutate at a good vertex. 
\end{definition}
As mentioned above, all maximal green sequences are good sequences. 
\begin{lemma}\label{lem:rotation_block}
Let $Q$ be a quiver. If $R \in \mut(Q)$ does not have a maximal green sequence then there is no maximal green sequence for $Q$ that passes through $R$. 
\end{lemma}
\begin{proof}
If there was such a maximal green sequence then by Theorem~\ref{thm:rotationlem} $R$ would have a maximal green sequence which contradicts our assumption. 
\end{proof}
\begin{lemma}\cite[Theorem 4]{brustle2}\label{lem:no_multi_head}
A maximal green sequence never mutates at the head of a multiple edge. 
\end{lemma}

\begin{corollary}\label{cor:max_are_good}
A maximal green sequence is a good sequence. 
\end{corollary}
\begin{proof}
This is a direct result of Theorem~\ref{thm:subquiver}, Lemma~\ref{lem:rotation_block},  and Lemma~\ref{lem:no_multi_head}.
\end{proof}
To obtain our result for rank 4 minimal mutation-infinite quivers it is sufficient to show that a particular family of quivers do not have a maximal green sequence. 

Let $Q$ be a quiver. The opposite quiver $Q^\text{op}$ is the quiver obtained from $Q$ by reversing all of the edges of $Q$. 
\begin{lemma}\label{lem:good_verticesR}
Let $R_{a,b,c}$ denote the rank 4 quiver given in Figure~\ref{fig:rank4_bad} with $b,c \geq 2$ and $a \leq c - 2$. The only good vertex of $R_{a,b,c}$ is vertex 3 and the only good vertex of $(R_{a,b,c})^\text{op}$ is vertex 2. 
%Furthermore, for any quiver of the form $R_{a,b,c}$ or $(R_{a,b,c})^\text{op}$ 
That is the only good vertex is the sink in the induced subquiver consisting of the vertices 1,2, and 3. 
%\todo[inline]{Do we really need to explicitly state anything about isomorphic quivers? }
\end{lemma}
\begin{proof}
If we mutate at vertex 1, then $\mu_1(R_{a,b,c})[\{2,3,4\}] \simeq Q_{2,c-a,b}$ which does not have a maximal green sequence by Theorem~\ref{thm:rank3_no_mgs} so vertex 1 is not a good vertex. The vertices 2 and 4 are not good vertices since they are heads of multiple arrows as $b,c \geq 2$. 

%The last statement follows from the observation that vertex 3 and vertex 2 are the sinks in the subquivers $R_{a,b,c}[\{1,2,3\}]$ and $(R_{a,b,c})^\text{op}[\{1,2,3\}]$, respectively. 
\end{proof}

\begin{figure}[t]
	\begin{tikzpicture}
	\coordinate (A) at (0,0);
	\coordinate (B) at (4,0);
	\coordinate (C) at (2,1.73);
	\coordinate (D) at (2,3.46);
	\draw[fill=gray] (A) circle (1.5pt);
	\draw[fill=gray] (B) circle (1.5pt);
	\draw[fill=gray] (C) circle (1.5pt);
	\draw[fill=gray] (D) circle (1.5pt);
	\draw[->, shorten >=5, shorten <=5, >=stealth'] (C) to(A);
	\draw[->, shorten >=5, shorten <=5, >=stealth'] (C) to (B);
	\draw[->, shorten >=5, shorten <=5, >=stealth'] (A) to (B);
	\draw[->, shorten >=5, shorten <=5, >=stealth'] (D) to node[above left, fill=white]{$a$}  (A) node[left]{1};
	\draw[->, shorten >=5, shorten <=5, >=stealth'] (B) node[right]{3} to node[above right, fill=white]{$c$}  (D);
	\draw[->, shorten >=5, shorten <=5, >=stealth'] (D) node[left]{4} to node[right]{$b$}  (C) node[below]{2};
	\end{tikzpicture}
	\caption{The rank $4$ quiver $R_{a,b,c}$ for $b,c \geq 2$ and $a \leq c - 2$. This
	quiver has only one good vertex, labelled $3$.}
	\label{fig:rank4_bad}
\end{figure}
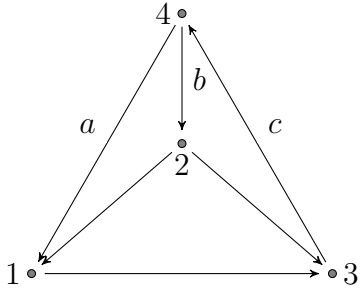

\begin{corollary}\label{cor:good_sequenceR}
Assume $c-a\geq 2$ and $|c-b-a| > 0$. \begin{enumerate}
\item If $c>b \geq 2$ and $\mgsi$ is a good sequence for $R_{a,b,c}$ of length $k$ we have
\[ \mu_{\mgsi}(R_{a,b,c})=\begin{cases} R_{a,b+n(c-b-a),c+n(c-b-a)} & \text{if } k = 2n,\\
(R_{c-b,c+n(c-b-a),b+(n+1)(c-b-a)})^\text{op} & \text{if } k = 2n+1.\end{cases}\]

\item If $b > c \geq 2$ and $\mgsi$ is a good sequence for $(R_{a,b,c})^\text{op}$ of length $k$ we have
\[ \mu_{\mgsi}((R_{a,b,c})^\text{op})=\begin{cases} (R_{a,b+n(b+a-c),c+n(b+a-c)})^\text{op} & \text{if } k = 2n,\\
(R_{b-c,c+(n+1)(b+a-c),b+n(b+a-c)})^\text{op} & \text{if } k = 2n+1.\end{cases}\]
\end{enumerate}
\end{corollary}
\begin{proof}
We only prove the case when $c>b \geq 2$ as the other case is analogous.

Assume $c>b \geq 2$. By Lemma~\ref{lem:good_verticesR} a good sequence must begin by mutating at 3. Now since $c>b$ we have $\mu_3(R_{a,b,c}) \simeq (R_{c-b,c-a,c})^\text{op}$. By assumption $c-a \geq 2$ so by Lemma~\ref{lem:good_verticesR} a good sequence must mutate at the sink in $\mu_3(R_{a,b,c})[\{1,2,3\}]$, which is vertex 1. Continuing our good sequence by mutating at vertex 1 we see that $$\mu_1\mu_3(R_{a,b,c}) \simeq R_{a,c-a,2c-b-a}.$$ 
Note that $2c-b-a > c-a \geq 2$, and $c-b-a>0$ so $2c-b-a>c$ and $2c-b-2a>c-a \geq 2$. That is, if we set $a^*=a, b^*=c-a = b + (c-b-a)$ and $c^* = 2c-b-a = c + (c-b-a)$ we see that $c^*-a^* \geq 2, c^*-b^* -a^* >0,$ and $c^*>b^* \geq 2$, so all of our hypotheses from the statement of the corollary are again satisfied. 

Repeating the argument above we know at each mutation step we have exactly one choice of vertex to mutate at for our mutation sequence to be a good sequence.  Continuing this process we see that if $\mgsi$ is a good sequence of length $k$, then we obtain the formula given in the statement of the result. 
\end{proof}

%\begin{corollary}\label{cor:good_sequenceR2}
%Assume $c-a\geq 2$ and $|c-b-a| > 0$. \begin{enumerate}
%\item If $c>b \geq 2$ a good sequence for $R_{a,b,c}$ never mutates at vertex 4.
%\item If $b > c \geq 2$ then a good sequence for $(R_{a,b,c})^\text{op}$ never mutates at vertex 4.
%\end{enumerate}  
%\end{corollary}

\begin{lemma}\label{lem:bad_rank4}
The quivers $R_{a,b,c}$ and $(R_{a,b,c})^\text{op}$ have no maximal green sequence when $$(a,b,c) \in \{(0,2,3),(1,4,3),(0,3,5),(2,5,4),(1,2,4)\}.$$
\end{lemma}

\begin{proof}
	For all three triples with $c > b \geq 2$, namely $\{(0,2,3),(0,3,5),(1,2,4)\}$
	we have $c-a>2$  and $c-b-a>0$ so Corollary~\ref{cor:good_sequenceR}(1)
	applies. If $\mgsi$ is a maximal green sequence for $R_{a,b,c}$ it is also a
	good sequence. However, if $\mgsi$ is a maximal green sequence then by
	Lemma~\ref{lem:iso_mgs} we have $\mu_{\mgsi}(R_{a,b,c}) = R_{a,b,c}$, which
	contradicts Corollary~\ref{cor:good_sequenceR}. Therefore $R_{a,b,c}$ has no
	maximal green sequence. By Lemma~\ref{lem:opposite_quiver} the opposite
	quivers $(R_{a,b,c})^\text{op}$ do not have a maximal green sequence. 

We may apply an identical argument using Corollary~\ref{cor:good_sequenceR}(2) and Corollary~\ref{cor:max_are_good} to the quivers $(R_{a,b,c})^\text{op}$ for $\{(1,4,3),(2,5,4)\}$ to show that they don't have a maximal green sequence. Then again we may apply Lemma~\ref{lem:opposite_quiver} to see that the quivers $R_{a,b,c}$ do not have a maximal green sequence.
\end{proof}

%\begin{proof} For all three triples with $c > b \geq 2$, namely $\{(0,2,3),(0,3,5),(1,2,4)\}$ we have $c-a>2$ so by Corollary~\ref{cor:good_sequenceR2}(1) any good sequence for $R_{a,b,c}$ will never mutate at vertex 4. By Corollary~\ref{cor:max_are_good} all maximal green sequences are good sequences, but since a maximal green sequence must mutate every vertex at least once, no maximal green sequence exists for $R_{a,b,c}$. By Lemma~\ref{lem:opposite_quiver} the opposite quivers $(R_{a,b,c})^\text{op}$ do not have a maximal green sequence. 
%
%We may apply an identical argument using Corollary~\ref{cor:good_sequenceR2}(2) and Corollary~\ref{cor:max_are_good} to the quivers $(R_{a,b,c})^\text{op}$ for $\{(1,4,3),(1,2,4)\}$ to show that they don't have a maximal green sequence. Then again we may apply Lemma~\ref{lem:opposite_quiver} to see that the quivers $R_{a,b,c}$ do not have a maximal green sequence.
%\end{proof}
\begin{theorem}\label{thm:rank4_eg}
Let $Q$ be a minimal mutation-infinite quiver of rank 4. Then $\Psi$ is a proper subgraph of $\Gamma$ and the connected component $\widehat{\Psi}$ of $\Psi$ that contains $Q$ is finite and contains the entire move-class of $Q$.
\end{theorem}
\begin{proof}
For any minimal mutation-finite quiver $Q$ it is easy to see that there exists a quiver $R$ in its mutation class that contains a rank 3 subquiver without a maximal green sequence. Therefore by Theorem~\ref{thm:subquiver} $R$ does not have a maximal green sequence so $\Psi \neq \Gamma$. 

The rest of the claim is verified via direct calculation. For each move-class representative of a rank 4 minimal-mutation infinite quiver given in Figure~\ref{fig:hcsreps} we compute its unlabelled exchange graph. For each new vertex added to the graph we test if the associated quiver has a maximal green sequence. In each case we obtain a component of $\Gamma$ that is bounded by quivers \begin{enumerate}
\item containing a rank 3 subquiver of the form to $Q_{a,b,c}$ with $a,b,c \geq 2$;
\item and the quivers $R_{a,b,c}$ and $(R_{a,b,c})^\text{op}$ for one of the triples considered in Lemma~\ref{lem:bad_rank4}.
\end{enumerate}
In the first case these quivers do not have a maximal green sequence by Theorem~\ref{thm:subquiver} and Theorem~\ref{thm:rank3_no_mgs}. The other two quivers do not have a maximal green sequence by Lemma~\ref{lem:bad_rank4}.

Upon inspection we see that the entire move class of the representative is contained in this component. 
\end{proof}

\begin{example}\label{ex:mgs-graph}
Let $Q$ be the 4th rank 4 quiver given in Figure~\ref{fig:hcsreps}. In Figure~\ref{fig:rank4_eg} we give $\widehat\Psi(Q)$ together with the vertices of $\Gamma(Q)$ that are adjacent to vertices of $\widehat\Psi(Q)$ to illustrate the boundedness of $\widehat\Psi(Q)$. The graph $\widehat\Psi(Q)$ consists of the quivers with a black frame. The quivers with a red frame do not have maximal green sequences and are not a part of $\widehat\Psi(Q)$.
\end{example}
\begin{figure}
\begin{center}
\includegraphics[scale=.6]{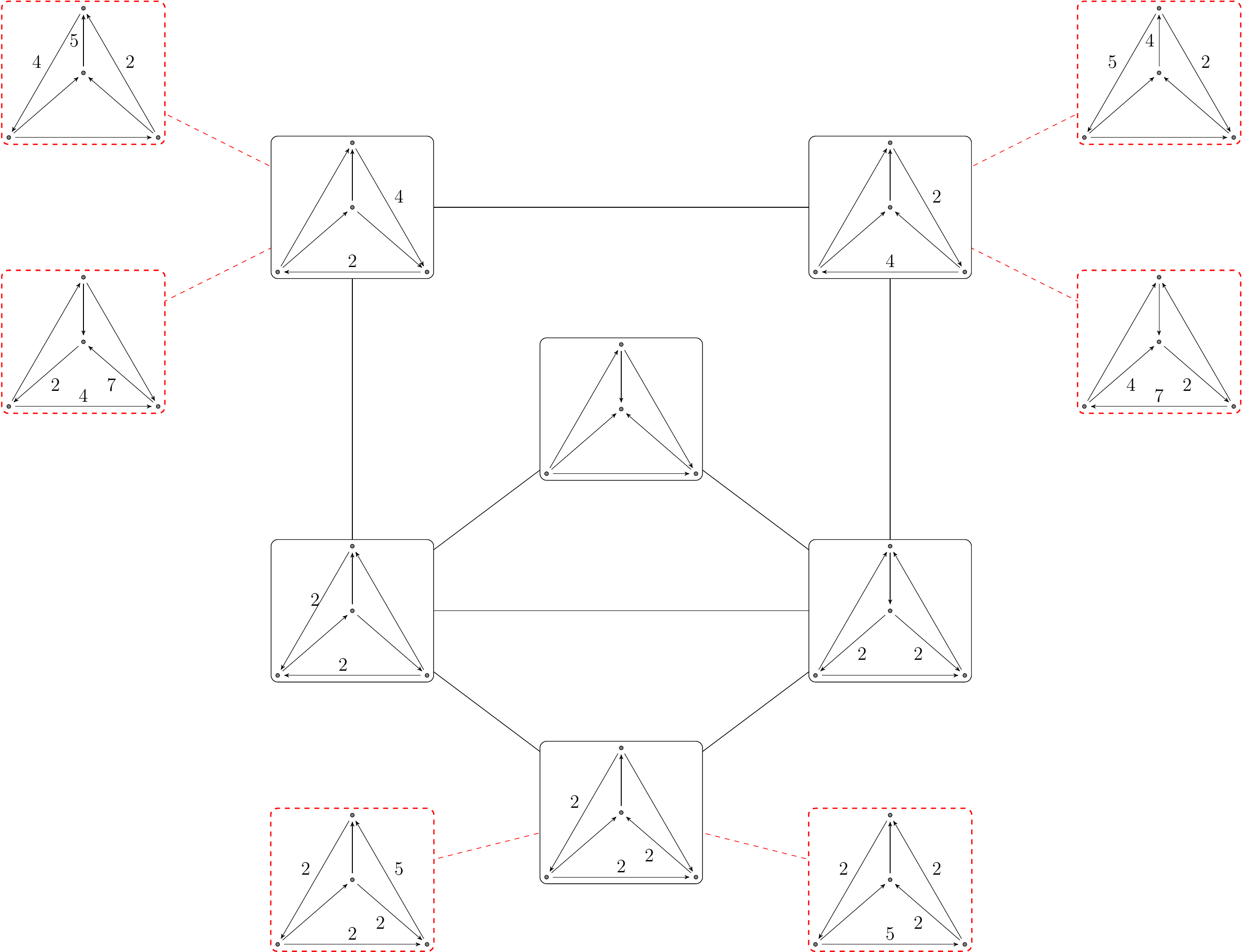}
\end{center}
\caption{A subgraph of the exchange graph of the 4th quiver $Q$ in
Figure~\ref{fig:hcsreps} containing $\widehat\Psi(Q)$, as described in
Example~\ref{ex:mgs-graph}.}
\label{fig:rank4_eg}
\end{figure}

\section{Other questions and conjectures}\label{sec:conj}

It is likely that an identical phenomenon occurs for higher rank minimal mutation-infinite quivers and that it can be shown using techniques similar to the ones presented here. It is straightforward to compute a candidate component for $\widehat\Psi$. The issue is that there are many more quivers that bound this region and it is a long process to show that they do not have maximal green sequences. Already in the case of the first rank 7 mutation class with a hyperbolic Coxeter representative, there are at least 1200 quivers to check that do not posses a rank 3 subquiver without a maximal green sequence. 
\begin{conjecture}
The results of Theorem~\ref{thm:rank4_eg} hold for all minimal mutation-infinite quivers. 
\end{conjecture}
We also believe that in these cases $\widehat{\Psi} = \Psi$. However we are unable to prove this. This would immediately follow from an affirmative answer to the following question. 
\begin{question}
Is the graph $\Psi$ a connected subgraph of $\Gamma$?
\end{question}
In light of Theorem~\ref{thm:directsum-mgs} we also think that it would be interesting to explore the relationship between $\Psi(Q)$ and $\Psi(Q')$ for two quivers $Q$ and $Q'$ and that of the graph $\Psi(Q \oplus Q')$. 
\begin{conjecture}
Suppose that $Q$ and $Q'$ are two quivers such that $\Psi(Q)$ and $\Psi(Q')$ are finite and non-empty, then $\Psi(Q \oplus Q')$ is finite and non-empty.
\end{conjecture}
One goal of this approach would be to answer the following question. 
\begin{question}
For any quiver $Q,$ are there finitely many quivers in $\mut(Q)$ that have a maximal green sequence?
\end{question}

%Uncomment to get old bibliography back
%\input{bib-backup}
\bibliographystyle{hplain}
\bibliography{bibliography}
\addresshere

\newpage
\appendix
\section*{Appendix of tables}
We provide tables of mutation invariants for the different move-classes of the minimal mutation infinite-quivers. We label the move-classes first by the rank of the quivers and then with a subscript referring to the order in which their representative appears in Figure~\ref{fig:hcsreps}, Figure~\ref{fig:dareps}, and  Figure~\ref{fig:excreps}. The starred values are conjectural. It is a well-known fact that the determinant of the matrix $B_Q$ is also invariant under mutation, but this invariant does not give us any new information about the mutation-classes so it is omitted.
\begin{table}[h]
\begin{tabular}{|c|c|c|c|}
\hline 
Move-class & rank$(B_Q)$  & Acyclic Quivers & Non-acyclic in $\widehat\Psi(Q)$ \\ 
\hline 
$4_1$ & 4 & 6 & 14 \\ 
\hline 
$4_2$ & 2  & 4 & 12 \\ 
\hline 
$4_3$ & 4  & 2 & 13 \\ 
\hline 
$4_4$ & 4  & 1 & 5 \\ 
\hline 
$4_5$ & 4 & 0 & 17 \\ 
\hline 
$4_6$ & 4 & 6 & 14 \\ 
\hline 
%\end{tabular} 
%\caption{}
%\end{table}
%
%\begin{table}[h]
%\begin{tabular}{|c|c|c|c|}
%\hline 
%• & rank$(B_Q)$ & $\det(B_Q)$ & Acyclic Quivers  \\
%\noalign{\vspace{6pt}}
%\hline 
$5_1$ & 4 & 8 & $80^*$ \\ 
\hline 
$5_2$ & 4  & 10 & $55^*$ \\ 
\hline 
$5_3$ & 4  & 5 & $101^*$ \\ 
\hline 
$5_4$ & 2  & 5 & $25^*$ \\ 
\hline 
\end{tabular} 
\caption{Rank 4 and 5 hyperbolic Coxeter simplex move-classes.}\label{tab1}
\end{table}
\begin{table}
\begin{tabular}{|c|c|c|c|}
\hline 
Move-class & rank$(B_Q)$ & Acyclic Quivers  \\
\hline 
$6_1$ & 4 \ & 16 \\ 
\hline 
$6_2$ & 2 & 6 \\ 
\hline 
$6_3$ & 6 & 10 \\ 
\hline 
$6_4$ & 6 & 20 \\ 
\hline 
$7_1$& 6 & 48 \\ 
\hline 
$7_2$ & 6 & 12 \\ 
\hline 
$7_3$ & 6 & 30 \\ 
\hline 
$7_4$ & 6 & 28 \\ 
\hline 
$8_1$ & 8 & 80 \\ 
\hline 
$8_2$ & 6 & 96 \\ 
\hline 
$8_3$ & 8 & 14 \\ 
\hline 
$8_4$ & 8 & 42 \\ 
\hline 
$8_5$ & 8 & 70 \\ 
\hline 
$9_1$ & 8 & 219 \\ 
\hline 
$9_2$ & 8 & 151 \\ 
\hline 
$9_3$ & 8 & 16 \\ 
\hline 
$9_4$ & 8 & 55 \\ 
\hline 
$9_5$ & 8 & 95 \\ 
\hline 
$9_6$ & 8 & 76 \\ 
\hline 
$10_1$ & 10  & 225 \\ 
\hline 
$10_2$ & 8 & 138 \\ 
\hline 
\end{tabular} 
\caption{Higher rank hyperbolic Coxeter simplex move-classes.}\label{tab2}
\end{table}
\begin{table}
\begin{minipage}{.48\textwidth}
\centering
\begin{tabular}{|c|c|}
\hline 
Move-class & rank$(B_Q)$   \\
\hline 
$6_5$ & 6 \\ 
\hline 
$6_6$ & 4 \\ 
\hline 
$7_5$ & 6 \\ 
\hline 
$8_6$ & 8  \\ 
\hline 
$9_7$ & 8 \\ 
\hline 
$10_3$ & 10 \\ 
\hline 
\end{tabular} 
\caption{Double arrow move-classes.}\label{tab3}
\end{minipage} 
\begin{minipage}{.48\textwidth}
\centering
\begin{tabular}{|c|c|}
\hline 
Move-class & rank$(B_Q)$    \\
\hline 
$7_6$ & 6  \\ 
\hline 
$8_7$ & 6  \\ 
\hline 
$8_8$ & 8  \\ 
\hline 
$9_8$ & 8 \\ 
\hline 
$9_9$ & 8  \\ 
\hline 
$9_{10}$ & 8\\ 
\hline 
$10_4$ & 10  \\ 
\hline 
$10_5$ & 10  \\ 
\hline 
$10_6$ & 8  \\ 
\hline 
$10_7$ & 10  \\ 
\hline 
\end{tabular} 
\caption{Exceptional move-classes.} \label{tab4}
\end{minipage}
\end{table}

\end{document}